\documentclass[psamsfonts]{amsart}

\usepackage{amssymb,amsfonts}
\usepackage{enumerate}
\usepackage{mathrsfs}
\usepackage{url}
\usepackage{mathtools}
\usepackage{tikz}
\usepackage[bookmarks=true]{hyperref}
\usepackage[font=small,labelfont=bf]{caption}
\usepackage{amssymb}
\usepackage[nameinlink,capitalize]{cleveref}

\newtheorem{thm}{Theorem}[section]
\newtheorem{cor}[thm]{Corollary}
\newtheorem{prop}[thm]{Proposition}

\newtheorem{lem}[thm]{Lemma}

\newtheorem{quest}[thm]{Question}

\theoremstyle{definition}
\newtheorem{defn}[thm]{Definition}

\newtheorem{exmp}[thm]{Example}

\newtheorem{fact}[thm]{Fact}

\newtheorem{claim}[thm]{Claim}

\theoremstyle{remark}
\newtheorem{rem}[thm]{Remark}

\newcommand{\tp}{\text{tp}}
\newcommand{\GS}{\mathrm{GS}}
\newcommand{\acl}{\operatorname{acl}}

\makeatletter
\let\c@equation\c@thm
\makeatother
\numberwithin{equation}{section}

\def\Ind{\setbox0=\hbox{$x$}\kern\wd0\hbox to 0pt{\hss$\mid$\hss} \lower.9\ht0\hbox to 0pt{\hss$\smile$\hss}\kern\wd0} 

\def\Notind{\setbox0=\hbox{$x$}\kern\wd0\hbox to 0pt{\mathchardef \nn=12854\hss$\nn$\kern1.4\wd0\hss}\hbox to 0pt{\hss$\mid$\hss}\lower.9\ht0 \hbox to 0pt{\hss$\smile$\hss}\kern\wd0} 

\def\ind{\mathop{\mathpalette\Ind{}}} 

\def\nind{\mathop{\mathpalette\Notind{}}} 

\title{Generic Stability Independence and Treeless theories}

\author{Itay Kaplan}
\address{Einstein Institute of Mathematics, Hebrew University of
Jerusalem, 91904, Jerusalem Israel.}
\email{kaplan@math.huji.ac.il}

\author{Nicholas Ramsey}
\address{Department of Mathematics, University of Notre Dame, 255 Hurley, Notre Dame, IN 46556}
\email{sramsey5@nd.edu}

\author{Pierre Simon}
\address{Dept.\ of Mathematics, University of California, Berkeley, 970 Evans 
Hall \#3840, Berkeley, CA 94720-3840 USA}
\email{simon@math.berkeley.edu}
\thanks{Kaplan would like to thank the Israel Science Foundation for their
support of this research (grants no. 1254/18 and 804/22).
Simon was partially supported by the NSF (grants no. 1665491 and 1848562).}
\date{\today}

\begin{document}

\begin{abstract}
We initiate a systematic study of \emph{generic stability independence} and introduce the class of \emph{treeless theories} in which this notion of independence is particularly well-behaved.  We show that the class of treeless theories contains both binary theories and stable theories and give several applications of the theory of independence for treeless theories.  As a corollary, we show that every binary NSOP$_{3}$ theory is simple.  
\end{abstract}

\maketitle

\setcounter{tocdepth}{1}
\tableofcontents

We introduce the class of \emph{treeless theories}.  These theories are defined in terms of a certain kind of indiscernible collapse which informally corresponds to the inability of the theory to code trees.  This approach carves out a natural model-theoretic setting that contains both the stable theories and the binary theories. We build on the study of generically stable partial types begun in \cite{MR4049222} to develop a theory of independence, called GS-independence, which allows us to establish the rudiments of a structure theory for this class. Although the genesis of this approach comes from NIP theories, we show that treelessness has strong consequences for the largely orthogonal setting of theories in the SOP$_{n}$ hierarchy. 

We begin, in \cref{partialtypes}, with a study of generically stable global partial types, as defined in \cite{MR4049222}.  We show that, in an arbitrary theory, every complete type over a set of parameters $A$ extends to a unique maximal global partial type which is generically stable over $A$.  This is then used to define $\mathrm{GS}$-independence:  $a$ is said to be $\mathrm{GS}$\emph{-independent} from $b$ over $A$ if $b$ satisfies $\pi|_{Aa}$, where $\pi$ is the maximal global partial type which is generically stable over $A$ and extends $\text{tp}(b/A)$. In \cref{GS-independence section} we study the properties of this independence relation in general and find that it satisfies many of the basic properties of independence relations. 

In order to define treeless theories, we introduce in \cref{treeless section}, a new kind of indiscernible tree, which we call a \emph{treetop indiscernible}.  The index structure in a treetop indiscernible is, in essence, the same as that of a strongly indiscernible tree, together with a predicate identifying the leaves of the tree.  We show that finite trees (in a language with symbols for the tree partial order, the lexicographic order, and the binary meet function) together with a predicate for the leaves form a Ramsey class and hence structures with this age give rise to a sensible notion of generalized indiscernible. In the tree $\omega^{\leq \omega}$, the set $\omega^{\omega}$ of leaves carries the structure of a dense linear order (under $<_{lex}$), but also carries considerably more structure induced by the tree structure. The treeless theories are defined in \cref{treeless section} to be those theories in which, in any treetop indiscernible, this additional structure on the leaves is irrelevant, that is, the sequence of tuples indexed by the leaves ordered lexicographically is an indiscernible sequence.  In \cref{sec:symmetry and base mon} we connect treelessness to the above-mentioned work on $\mathrm{GS}$-independence, showing that, in treeless theories, GS-independence is symmetric and satisfies base monotonicity. 


In \cref{Stable section}, we prove that all stable theories are treeless and then in the remaining sections, we explore the consequences treelessness has for the SOP$_{n}$ hierarchy. In \cref{simple section}, we prove that NSOP$_1$ treeless theories are simple.  We obtain this result as a rapid consequence of the fact that GS-independence and Kim-independence coincide over models in NSOP$_{1}$ theories, but we also give an alternative argument for the corollary that binary NSOP$_{1}$ theories are simple, using only tools from the theory of Kim-independence, which may be of independent interest. In \cref{nsop3 section}, we show that every treeless NSOP$_{3}$ theory with indiscernible triviality is NSOP$_{2}$.  These hypotheses are met by any binary NSOP$_{3}$ theory and therefore, modulo Mutchnik's recent result \cite{mutchnik2022nsop} that NSOP$_{1}$ = NSOP$_{2}$, our results establish that every binary NSOP$_{3}$ theory is simple. This means, for example, that the known classification for binary homogeneous structures due to \cite{Koponen} applies directly to the \emph{a priori} much broader class of homogeneous binary NSOP$_{3}$ structures.

\section{Generically stable partial types} \label{partialtypes}

In the following two subsections, we recall definitions and basic properties of generically stable partial types from \cite{MR4049222}. The main result of the section is Corollary \ref{maximal extension}, which entails that every complete type over a set $A$ has a unique maximal extension to a global partial type which is generically stable over $A$.  This will serve as the basis of a notion of independence introduced in Section \ref{GS-independence section}. 

\subsection{ind-definable partial types}

We will work in a monster model $\mathbb{M}$ of a fixed complete theory $T$.  A partial type $\pi(x)$ (over $\mathbb{M}$) is a consistent set of formulas with parameters in $\mathbb{M}$ closed under finite conjunctions and logical consequences, that is:

$\bullet$ $\phi(x), \psi(x)\in \pi \Longrightarrow \phi(x)\wedge \psi(x) \in \pi$;

$\bullet$ $\phi(x)\in \pi \wedge \mathbb{M} \vDash \phi(x) \to \psi(x) \Longrightarrow \psi(x)\in \pi$.

Given a set $A$ of parameters, $\pi|_A$ or $\pi | A$ denotes the partial type obtained by taking the subset of $\pi$ composed of formulas with parameters in $A$. Note that, because we require $\pi$ to be closed under logical consequence, if $a\vDash \pi|_A$ then $\pi \cup \mathrm{tp}(a/A)$ is consistent.

A partial type $\pi$ is $A$\emph{-invariant} if it is invariant under automorphisms of $\mathbb{M}$ fixing $A$ pointwise.

\begin{defn}
We say that a partial type $\pi$ is \emph{ind-definable} over $A$ if for every $\phi(x;y)$, the set $\{b : \phi(x;b)\in \pi\}$ is ind-definable over $A$ ({\it i.e.}, is a union of $A$-definable sets).
\end{defn}

As noted in \cite[Section 2]{MR4049222}, one can represent an $A$-ind-definable partial type as a collection of pairs $$(\phi_i(x;y),d\phi_i(y)),$$ where $\phi_i(x;y)\in L$, $d\phi_i(y)\in L(A)$ such that $\pi(x)$ is equal to $\bigcup_i \{\phi_i(x;b) : b\in d\phi_i(\mathbb{M})\}$ (the same formula $\phi(x;y)$ can appear infinitely often as $\phi_i(x;y)$). And, conversely, given a family of pairs $(\phi_i(x;y),d\phi_i(y))$, if the partial type $\pi(x)$ generated by $\bigcup_i \{\phi_i(x;b) : b\in d\phi_i(\mathbb{M})\}$ is consistent, then it is ind-definable. Observe that the partial types $(\phi(x;y),d\phi(y))$ and $(d\phi(y)\to \phi(x;y) ; y=y)$ are the same.

\begin{fact}\label{lem_def} \cite[Lemma 2.2]{MR4049222}
Let $\pi(x)$ be a partial $A$-invariant type. Then $\pi$ is ind-definable over $A$ if and only if the set $X=\{(a,\overline{b}) : \overline{b}\in \mathbb{M}^{\omega}, a\vDash \pi|A\bar b\}$ is type-definable over $A$.
\end{fact}

Let $\pi(x)$ and $\eta(y)$ be two $A$-invariant partial types, where $\pi$ is ind-definable over $A$. Then there is an $A$-invariant partial type $(\pi \otimes \eta)(x,y)$ such that $(a,b)\vDash \pi \otimes \eta$ if and only if $b\vDash \eta$ and $a\vDash \pi |\mathbb{M} b$. Indeed, $(\pi \otimes \eta)(x,y)$ is generated by $\eta(y)$ along with pairs $(d\phi(y,z)\to \phi(x;y,z), z=z)$ (with $\phi \in L$ and $d\phi \in L(A)$), where the partial type $(\phi(x;y,z),d\phi(y,z))$ is in $\pi(x)$. If in addition $\eta$ is ind-definable over $A$, then so is $\pi \otimes \eta$. As usual, we define inductively $\pi^{(1)}(x_0) = \pi(x_0)$ and \[\pi^{(n+1)}(x_0,\ldots,x_{n})=\pi(x_{n})\otimes \pi^{(n)}(x_0,\ldots,x_{n-1}).\] Also set \[\pi^{(\omega)}(x_0,x_1,\ldots) = \bigcup_{n<\omega} \pi^{(n)}(x_0,\ldots,x_{n-1}).\] All those types are ind-definable over $A$.

\smallskip

Instead of a partial type $\pi$, one could also consider the dual ideal $I_\pi$ of $\pi$ defined as the ideal of formulas $\phi(x)$ such that $\neg \phi(x)\in \pi$. Then an $I_\pi$-wide type (namely a type not containing a formula in $I_\pi$) is precisely a type over some $A$ containing $\pi|A$.

\subsection{Generic stability}

\begin{defn}
Let $\pi(x)$ be a partial type. We say that $\pi$ is \emph{generically stable over} $A$ if $\pi$ is ind-definable over $A$ and the following holds:

(GS) if $(a_k:k<\omega)$ is such that $a_k \vDash \pi|Aa_{<k}$ and $\phi(x;b)\in \pi$, then for all but finitely many values of $k$, we have $\vDash \phi(a_k;b)$.
\end{defn}


\begin{defn}
We say that a partial type $\pi(x)$ over $\mathbb{M}$ is \emph{finitely satisfiable in} $A$ if any formula in it has a realization in $A$ (recall that we assume $\pi$ to be closed under conjunctions). 
\end{defn}

The following facts record some basic properties of generically stable partial types:

\begin{fact}\label{lem_deffs} \cite[Lemma 2.4]{MR4049222}
Let $\pi$ be a partial type ind-definable over $A$. Let $a\vDash \pi|A$ and $b$ such that $\mathrm{tp}(b/Aa)$ is finitely satisfiable in $A$. Then $a\vDash \pi|Ab$.
\end{fact}

\begin{fact} \label{forking fact}  \cite[Proposition 2.6]{MR4049222}
Let $\pi$ be a partial type generically stable over $A$. Then:

(FS) $\pi$ is finitely satisfiable in every model containing $A$;

(NF) let $\phi(x;b)\in \pi$ and take $a\vDash \pi|A$ such that $\vDash \neg \phi(a;b)$. Then both $\mathrm{tp}(b/Aa)$ and $\mathrm{tp}(a/Ab)$ fork over $A$.
\end{fact}

%
%
%
%
%
%

\begin{fact}\label{lem_reducepi} \cite[Lemma 2.9]{MR4049222}
Let $\pi(x)$ be generically stable over $A$ and let $\pi_0(x)\subseteq \pi(x)$ be a partial ind-definable type, definable over some $A_0\subseteq A$. Then there is $\pi_*(x)\subseteq \pi(x)$ containing $\pi_0(x)$ which is generically stable and defined over some $A_*\subseteq A$ of size $\leq |A_0|+|T|$.
\end{fact}

The following lemma is new, but is a strengthening of \cite[Lemma 2.11]{MR4049222}:

\begin{lem}\label{lem:gen_stable_amalgam}
Let $\pi(x), \lambda(x)$ be two partial types ind-definable over $A$. Assume that $\lambda$ is generically stable over $A$ and that $\pi(x)|_A \cup \lambda(x)|_A$ is consistent. Then $\pi(x)\cup \lambda(x)$ is generically stable over $A$. 
\end{lem}
\begin{proof}
We show by induction on $n<\omega$ that there is $\bar a=(a_i:i<n)$ such that $\bar a\vDash \pi^{(n)}(x)|_A$ and $\bar a^* \vDash \lambda^{(n)}(x)|_A$, where $\bar a^* = (a_{n-1},a_{n-2},\ldots,a_0)$. For $n=1$, this is the hypothesis. Assume we know it for $n$, witnessed by $\bar a=(a_i:i<n)$. Since $\pi(x)|_A \cup \lambda(x)|_A$ is consistent, so is $\pi(x)|_A \cup \lambda(x)$. Let $\bar b=(b_i:i<\kappa)$ be a long Morley sequence in that partial type. Since we assume our partial types are closed under logical consequence, the fact that $\overline{a} \vDash \pi^{(n)}|_{A}$ implies that $\pi^{(n)} \cup \text{tp}(\overline{a}/A)$ is consistent.  Thus, composing by an automorphism over $A$, we may assume that $\bar a\vDash \pi^{(n)}|_{A\bar b}$. By generic stability of $\lambda$, there is $i<\kappa$ such that $b_i\vDash \lambda|_{A\bar a}$. It follows that $(b_i)^{\frown}\bar a \vDash \pi^{(n+1)}|_A$ and that ${\overline{a}^{*}}^{\frown} (b_i) \vDash \lambda^{(n+1)}|_A$. This finishes the induction.

This being done, we can construct, by Fact \ref{lem_def} and compactness, a sequence $\bar d=(d_i:i<\omega)$ which is a Morley sequence of $\pi$ over $A$ such that the sequence in the reverse order is a Morley sequence of $\lambda$ over $A$. We can further assume that $\bar d\vDash \pi^{(\omega)}|_\mathbb{M}$. The set of formulas over $\mathbb{M}$ that are true on almost all elements of $\bar d$ contains $\lambda(x)$ and therefore $\pi(x)\cup \lambda(x)$ is consistent.

Finally, we conclude that $\pi(x) \cup \lambda(x)$ is generically stable over $A$.  Let $\mu(x)$ be the partial type generated by $\pi(x) \cup \lambda(x)$.  It is clear that $\mu(x)$ is ind-definable over $A$ using \cref{lem_def} and the fact that $\{(a,\overline{b}) : \overline{b} \in \mathbb{M}^{\omega}, a \vDash \mu(x)|_{A\overline{b}}\}$ is equal to the intersection $\{(a,\overline{b}) : \overline{b} \in \mathbb{M}^{\omega}, a \vDash \pi(x)|_{A\overline{b}}\} \cap \{(a,\overline{b}) : \overline{b} \in \mathbb{M}^{\omega}, a \vDash \lambda(x)|_{A\overline{b}}\}$.  If $\varphi(x;b) \in \mu(x)$, then there are $\psi_{0}(x;c) \in \pi(x)$ and $\psi_{1}(x;d) \in \lambda(x)$ such that $\psi_{0}(x;c) \wedge \psi_{1}(x;d) \vdash \varphi(x;b)$.  Taking $I = (a_{i} : i < \omega) \vDash \mu^{(\omega)}|_{A}$, since $I$ is Morley over $A$ in both $\pi$ and $\lambda$, we know that both $\{i : \vDash \psi_{0}(a_{i};c)\}$ and $\{i : \vDash \psi_{1}(a_{i};d)\}$ are cofinite so $\{i : \vDash \varphi(a_{i};b)\}$ is cofinite as well.  This shows $\mu$ is generically stable over $A$.  

\end{proof}

\begin{cor} \label{maximal extension}
Let $p(x)\in S(A)$.  There is a unique maximal global partial type $\pi_{p}$ generically stable over $A$ consistent with $p$\textemdash that is, if $\pi$ is a global generically stable partial type consistent $p$, then $\pi \subseteq \pi_{p}$.  It follows, in particular, that $\pi_{p}$ extends $p$. 
\end{cor}

\begin{proof}
By \cref{lem:gen_stable_amalgam}, if $\pi(x)$ and $\lambda(x)$ are two generically stable partial types consistent with $p$, ind-definable over $A$, then $\pi(x)\cup \lambda(x)$ is consistent and even generically stable over $A$. Hence we can define $\pi_p(x)$ as the union of all generically stable partial types consistent with $p$ and ind-definable over $A$. Then $\pi_p(x)$ is consistent and is the maximal $A$-invariant generically stable partial type consistent with $p$.  As $p$ itself is generically stable over $A$, it follows that $\pi_{p}$ extends $p$. 
\end{proof}

\begin{lem} \label{equiv rel GS-forks}
Suppose $p(x)$ is a complete type over $A$ and $E(x,y)$ is an equivalence relation which is $\bigvee$-definable over $A$ and has unboundedly many classes represented by realizations of $p$. If $\pi \supseteq p$ is the maximal generically stable partial type over $A$ extending $p$, then $\pi \vdash \neg E(x,c)$ for all $c \in \mathbb{M}$.  
\end{lem}

\begin{proof}
Let $\pi_{0}(x)$ be the global partial type defined by 
$$
\pi_{0}(x) = p(x) \cup \bigcup \{\neg E(x;c) : c \in \mathbb{M}\},
$$
which is a consistent partial type by our assumption that $E(x,y)$ is $\bigvee$-definable and has unboundedly many classes among realizations of $p$. We have $\pi_{0}$ is ind-definable over $A$ since, writing $E(x,y) = \bigvee \psi_{i}(x,y)$, we can ind-define $\pi_{0}$ via the schema $(\varphi(x), y=y)_{\varphi(x) \in p}$ and $(\neg \psi_{i}(x,y), y = y)_{i}$.   If $(a_{i} : i < \omega)$ is a sequence with $a_{i} \vDash \pi_{0}|_{Aa_{<i}}$, then we have $\neg E(a_{i},a_{j})$ for all $i \neq j$.  Therefore, if $c \in \mathbb{M}$, then $c$ can be $E$-equivalent to at most one $a_{i}$.  Therefore, if $\chi(x,c) \in \pi_{0}$, then we have $\vDash \chi(a_{j},c)$ for all but at most one $j$.  This shows $\pi_{0}$ is a generically stable partial type over $A$ and is therefore contained in the maximal one extending $p$ by \cref{maximal extension}. \end{proof}

The following proposition is essentially \cite[Remark 6.13]{MR4049222}:

\begin{prop}\label{prop:gen_stable_restriction}
Let $\pi(x,y)$ be generically stable over $A$. Then the partial type $\eta(x) = (\exists y)\pi(x,y)$ (which is also the restriction of $\pi$ to the $x$ variable) is generically stable over $A$.
\end{prop}
\begin{proof}
Note that for any set $B\supseteq A$, $\pi | B = (\exists y)(\pi(x,y)|_B)$.

Since $\pi(x,y)$ is $A$-invariant, $\eta(x)$ is also $A$-invariant. We first show that $\eta$ is ind-definable using \cref{lem_def}. Fix a variable $\bar z$ and let $X(x,y,\bar z)$ be the set of triples $\{(a,b,\bar c) :(a,b) \vDash \pi|A\bar c\}$. For any tuples $a$ and $\bar c$, we have $a\vDash \eta|A\bar c$ if and only if there is $b$ such that $(a,b,\bar c)\in X$. As $X$ is type-definable by \cref{lem_def}, this whole condition is type-definable. By one more application of \cref{lem_def}, $\eta$ is ind-definable.

We next show (GS). Assume for a contradiction that for some $\phi(x;c)\in \eta$, the set $\eta^{(\omega)}(x_k:k<\omega) \cup \{\neg \phi(x_k;c) : k<\omega\}$ is consistent. Let $(a_k)_{k<\omega}$ realize it. Note that if we replace $(a_k:k<\omega)$ by a sequence $(a'_k:k<\omega)$ which has the same type over $A$, then we can find $c'\equiv_A c$ such that $\neg \phi(a'_k;c')$ holds for all $k$. By invariance of $\eta$, we have $\phi(x;c')\in \eta$, so $(a'_k : k < \omega)$ also witnesses a failure of (GS).

We build by induction on $k$ tuples $(b_k:k<\omega)$ such that $\mathrm{tp}(a_k,b_k/A)=\mathrm{tp}(a,b/A)$ and $(a_k,b_k)\vDash \pi | Aa_{<k}b_{<k}$. We can find $b_0$ since $a_0\vDash \eta|A$. Assume we have found $b_k$. As $a_{k+1}\vDash \eta|Aa_{\leq k}$, there is an automorphism $\sigma$ fixing $Aa_{\leq k}$ such that $\sigma(a_{k+1})\vDash \eta | Aa_{\leq k}b_{\leq k}$. By the remark above, we may replace the sequence $a_{>k}$ by $\sigma(a_{>k})$, since this does not alter the type of the full sequence $(a_i)_{i<\omega}$. Hence we may assume that actually $a_{k+1}\vDash \eta|Aa_{\leq k}b_{\leq k}$ and then we find $b_{k+1}$ as required.

We now have a sequence $(a_kb_k:k<\omega)$ such that $(a_k,b_k)_{k<\omega}\vDash \pi^{(\omega)}(x_k: k<\omega)$ and $c$ such that $\phi(x;c)\in \pi$ and $\neg \phi(a_k;c)$ holds for all $k$. Since the condition $(a_k,b_k)_{k<\omega}\vDash \pi^{(\omega)}(x_k :k<\omega)|_{A}$ is type definable by \cref{lem_def}, we can apply Ramsey and compactness and assume that the sequence $(a_kb_k:k<\omega)$ is indiscernible over $Ac$. Using (GS) for the type $\pi$, we conclude that for every $k$, $(a_k,b_k)\vDash \pi|Ac$. But by the definition of $\eta$, this means that $a_k\vDash \eta|Ac$. Contradiction.
\end{proof}

The following corollary is \cite[Proposition 2.13]{MR4049222}.  It follows immediately from \cref{prop:gen_stable_restriction}. 

\begin{cor}\label{prop_gsquant}
Let $\alpha(y)$ be a partial type, generically stable over $A$. Fix some $a,b\in \mathbb{M}$, $b\vDash \alpha(y)|_A$ and let $\rho(x,y)\subseteq \mathrm{tp}(a,b/A)$. Then the partial type $\pi(x):= (\exists y) (\alpha(y) \wedge \rho(x,y))$ is generically stable over $A$.
\end{cor}

\section{GS-independence} \label{GS-independence section}

We write $a\ind^{\mathrm{GS}}_A b$ if for every partial type $\pi(x)$ generically stable over $A$, if $b\vDash \pi|_A$, then $b\vDash \pi|_{Aa}$. Note that this is equivalent to saying that $b \vDash \pi_*|_{Aa}$, where $\pi_*$ is the maximal $A$-invariant generically stable partial type extending $\mathrm{tp}(b/A)$.  If $p$ is a partial type, we say that $p$ $\mathrm{GS}$\emph{-forks} over $A$ if there is some $B$ such that there is no $a \vDash p$ with $a \ind^{\mathrm{GS}}_{A} B$. 

\begin{lem} \label{lem:nf implies GS}
    If $a \ind^{f}_{A} b$ or $b \ind^{f}_{A} a$, then $a \ind^{\mathrm{GS}}_{A} b$. 
\end{lem}

\begin{proof}
   Immediate by \cref{forking fact}. 
\end{proof}

\begin{thm} \label{GS props}
The relation $\ind^{\mathrm{GS}}$ satisfies:

\begin{enumerate}
\item (invariance) If $A\ind^{\mathrm{GS}}_C B$ and $\sigma\in \mathrm{Aut}(\mathbb{M})$, then $\sigma(A)\ind^{\mathrm{GS}}_{\sigma(C)} \sigma(B)$.

\item (normality) If $A \ind^{\mathrm{GS}}_{C} B$, then $AC \ind^{\mathrm{GS}}_{C} BC$. 

\item (monotonicity) If $A\ind^{\mathrm{GS}}_C B$, $A'\subseteq A$, $B'\subseteq B$, then $A'\ind^{\mathrm{GS}}_C B'$.

\item (left and right existence) For all $A$ and $B$, $A \ind^{\mathrm{GS}}_{B} B$ and $A \ind^{\mathrm{GS}}_{A} B$.   

\item (right and left extension) If $A\ind^{\mathrm{GS}}_C B$ and $B'\supseteq B$, then there is $A' \equiv_{BC} A$ such that $A'\ind^{\mathrm{GS}}_C B'$. Similarly, if $A' \supseteq A$, then there is $B'\equiv_{AC} B$ such that $A'\ind^{\mathrm{GS}}_C B'$. 

\item (finite character) We have $A\ind^{\mathrm{GS}}_C B$ if and only if for all finite $A_0\subseteq A$ and $B_0\subseteq B$, we have $A_0\ind^{\mathrm{GS}}_C B_0$.

\item (left transitivity) If $C \subseteq B \subseteq A$, $B \ind^{\mathrm{GS}}_{C} D$, and $A \ind^{\mathrm{GS}}_{B} D$, then $A \ind^{\mathrm{GS}}_{C} D$. 

\item (local character on a club) For every finite tuple $a$ and for every set of parameters $B$, there is a club $\mathcal{C} \subseteq [B]^{\leq |T|}$ such that $a\ind^{\mathrm{GS}}_C B$ and  $a\ind^{\mathrm{GS}}_C B$ for all $C \in \mathcal{C}$.
\item (anti-reflexivity) We have $a\ind^{\mathrm{GS}}_C a$ if and only if $a\in \mathrm{acl}(C)$.

\item (algebraicity)\footnote{See also \cref{cor:alg base mon} to complete the picture.} If $a \ind^{\GS}_A b$ then $a \ind^{\GS}_A \acl(b)$ and $\acl(a) \ind^{\GS}_A b$. 

\end{enumerate}
\end{thm}
\begin{proof}
Invariance is clear from the definition.  The implication from $A \ind^{\mathrm{GS}}_{C} B$ to $AC \ind^{\mathrm{GS}}_{C} B$ is also clear from the definition, and the statement of normality follows from this by extension.  Monotonicity follows from the fact that adding dummy variables to a generically stable partial type preserves generic stability.

Existence (on both sides) follows directly from \cref{lem:nf implies GS} since clearly $A \ind^f_A B$ and $B \ind^f_B A$.

To prove right extension, assume that $A\ind^{\mathrm{GS}}_C B$ and let $B'= B\cup B''$. Let $\pi(x\hat{~}x'')$ be the unique maximal global partial type consistent with $\mathrm{tp}(BB''/C)$ which is generically stable over $C$. We need to show that $\mathrm{tp}(B/CA)\cup \pi|_{CA}$ is consistent. By \cref{prop:gen_stable_restriction}, the partial type $(\exists x'')\pi(x\hat{~}x'')$ is generically stable over $C$. It is therefore consistent with $\mathrm{tp}(B/CA)$ and the result follows, by \cref{maximal extension}. Left extension follows by definition: if $B \vDash \pi|_{AC}$ for $\pi$ generically stable over $C$, then $\pi|_{CA'} \cup \tp(B/AC)$ is consistent, so let $B'$ realize it.

Finite character on the left follows from the definition. To see finite character on the right, assume that we have $A \nind^{\mathrm{GS}}_C B$. Then there is a generically stable partial type $\pi(x)$ extending $\mathrm{tp}(B/C)$ and a formula $\phi(x)\in \pi|_{AC}$ such that $B\vDash \neg \phi$. The formula $\phi$ only involves a finite subset $B_0\subseteq B$. Write $B = B_0\cup B'$ and correspondingly split the variable $x = x_0\hat{~}x'$. By \cref{prop:gen_stable_restriction}, the partial type $\pi_0(x_0) := (\exists x')\pi(x_0 \hat{~} x')$ is generically stable over $C$. Then the formula $\phi(x)$ is a consequence of $\pi_0$ and we see that $B_0$ does not satisfy $\pi_0 |_{AC}$. Hence $A \nind^{\mathrm{GS}}_C B_0$.

Next, we consider left transitivity.  We will assume $a \ind^{\mathrm{GS}}_{Cb} d$ and $b \ind^{\mathrm{GS}}_{C} d$.  Let $\pi \supseteq \text{tp}(d/C)$ denote the maximal global partial type that is generically stable over $C$ and let $\tilde{\pi} \supseteq \text{tp}(d/Cb)$ denote the maximal global partial type that is generically stable over $Cb$. We want to show $d \vDash \pi|_{Cab}$, so pick $\phi(x;a,b) \in \pi$ and we will show that $\vDash \phi(d;a,b)$.  By our assumption that $b \ind^{\mathrm{GS}}_{C} d$, we know that $\text{tp}(d/Cb) \cup \pi$ is consistent.  It is also clearly generically stable over $Cb$, hence contained in $\tilde{\pi}$.  Thus, $\phi(x;a,b) \in \tilde{\pi}$ and the fact that $a \ind^{\mathrm{GS}}_{Cb} d$ entails that $\vDash \phi(d;a,b)$ as desired. 

We now prove local character on a club.  By \cref{lem:nf implies GS}, if $B \ind^f_{C} a$, then $a \ind^{\GS}_C B$ and $B \ind^{\GS}_C a$.  In particular this happens if $\text{tp}(B/aC)$ is finitely satisfiable in $C$. Therefore, it suffices to show that the set $\mathcal{C}$ defined by 
$$
\mathcal{C} = \{C \subseteq B : |C| \leq |T| \text{ and }\mathrm{tp}(B/aC)\text{ is finitely satisfiable in }C\},
$$
is a club of $[B]^{\leq |T|}$.  The set $\mathcal{C}$ is clearly closed under unions of chains of length $\leq |T|$, so we show it is unbounded.  Pick any $X \in [B]^{\leq |T|}$.  Inductively, we will build a sequence of sets $(C_{i})_{i < \omega}$ such that, for all $i < \omega$, we have the following: 
\begin{itemize}
\item $X \subseteq C_{i} \subseteq C_{i+1} \subseteq B$.
\item $|C_{i}| \leq |T|$.
\item If $\varphi(x;y) \in L(C_{i})$ and there is some $b \in B$ with $\vDash \varphi(b;a)$, then there is some $b' \in C_{i+1}$ with $\vDash \varphi(b';a)$.  
\end{itemize}
There is no problem in carrying out the induction:  we begin with $C_{0} = X$, and since $|C_{i}| \leq |T|$, there are only $|T|$ many formulas $\varphi(x;a)$ realized by some tuple in $B$ and we form $C_{i+1}$ by adding to $C_{i}$ one tuple from $B$ for each such formula. Then we put $C = \bigcup_{i} C_{i}$. By construction, $\text{tp}(B/Ca)$ is finitely satisfiable in $C$ and hence $X \subseteq C \in \mathcal{C}$.  


For anti-reflexivity, note that $\{x\neq b : b\in \mathbb{M}\}$ is a generically stable partial type, consistent with $\mathrm{tp}(a/C)$ if $a\notin \mathrm{acl}(C)$. Therefore $a\ind^{\mathrm{GS}}_C a$ implies that $a\in \mathrm{acl}(C)$. For the other direction, suppose $a \in \mathrm{acl}(C)$ and let $A$ be the finite set of realizations of $\mathrm{tp}(a/C)$.  By extension, there is $A' \equiv_{C} A$ such that $A' \ind^{\mathrm{GS}}_{C} A$ but, as a set, we must have $A = A'$ so $a \ind^{\mathrm{GS}}_{C} a$ follows by monotonicity.

Algebraicity: suppose that $a \ind^{GS}_A b$. The fact that $a \ind_A \acl(b)$ follows by right extension and invariance. Similarly, $\acl(a) \ind^{\GS}_A b$ follows from left extension and invariance.
\end{proof}

\begin{rem}
The form of local character in (5) was first isolated for Kim-independence in NSOP$_{1}$ theories in \cite{24-Kaplan2017}.  It, of course, implies the usual formulation of local character but is a more suitable analogue of the local character of non-forking independence in simple theories for contexts without base monotonicity.  Additionally, the proof of local character plus \cref{forking fact} imply local character on the left, since finite satisfiability implies non-forking.  That is, the proof establishes that for every finite tuple $a$ and set $B$, there is a club $\mathcal{C} \subseteq [B]^{\leq |T|}$ such that $B \ind^{\mathrm{GS}}_{C} a$ for all $C \in \mathcal{C}$. 
\end{rem}

Consider the following property:

\begin{description}
\item[(P)] If $\pi(x)$ is generically stable, then so is $\pi^{(\omega)}(x_0, x_1,\ldots)$.
\end{description}

\begin{prop} \label{P implies sym}
Assume that (P) holds, then $\ind^{\mathrm{GS}}$ satisfies symmetry: for any $A,a,b$ we have \[ a\ind^{\mathrm{GS}}_A b \iff b\ind^{\mathrm{GS}}_A a.\]
\end{prop}
\begin{proof}
Assume that $a\ind^{\mathrm{GS}}_A b$, but $b\nind^{\mathrm{GS}}_A a$. Let $\pi(x)$ be generically stable over $A$, consistent with $\mathrm{tp}(a/A)$, but not $\mathrm{tp}(a/Ab)$. Let $\phi(x,y)\in \mathrm{tp}(a,b/A)$ be such that $\neg\phi(x,b)\in \pi|Ab$. Let $n<\omega$ be maximal such that there is $(a_1,\ldots,a_n)\vDash \pi^{(n)}$ with $\bigwedge_{i\leq n} \phi(a_i,b)$. (Note that such an $n$ exists by generic stability and ind-definability.) Consider the partial type \[\eta_n(y) = \mathrm{tp}(b/A) \wedge (\exists (x_1,\ldots,x_n) \vDash \pi^{(n)}) \bigwedge_{i\leq n} \phi(x_i,y).\]
This type is generically stable by property (P) and \cref{prop_gsquant} and it is consistent with $\mathrm{tp}(b/A)$ by definition. As $a\ind^{\mathrm{GS}}_A b$, it is consistent with $\mathrm{tp}(b/Aa)$. But this means that we can find $a_1,\ldots,a_n\vDash \pi^{(n)}|_{Aa}$ with $\bigwedge_{i\leq n} \phi(a_i,b)$. But then $(a_0:=a,a_1,\ldots,a_n)\vDash \pi^{(n+1)}|_{A}$ and $\bigwedge_{i< n} \phi(a_i,b)$ holds. This contradicts the maximality of $n$.
\end{proof}

\begin{rem}
 In \cite[Example 2.12]{MR4049222}, there is an example which shows that property $P$ does not hold in general for generically stable partial types. 
\end{rem}


\begin{quest}\label{que:symmetry}
Is $\ind^{\mathrm{GS}}$ symmetric in general? Does it always satisfy transitivty on the right?
\end{quest}

\section{Treeless theories} \label{treeless section}

In this section, we define the treeless theories.  We begin by showing that \emph{treetop indiscernibles}, defined in the first subsection, have the modeling property.  Then we define treelessness in terms of a form of indiscernible collapse from the structure on the leaves of the treetop indiscernible to an indiscernible sequence.  

\subsection{Generalized indiscernibles and Ramsey classes}

In this subsection, we will define generalized indiscernibles and introduce a new kind of indiscernible tree, which allow us later on to define the treeless theories. 

\begin{defn}  Suppose $I$ is an $L'$-structure, where $L'$ is some language. 
\begin{enumerate}
\item  We say $(a_{i} : i \in I)$ is a set of $I$\emph{-indexed indiscernibles} if whenever 

$(s_{0}, \ldots, s_{n-1})$, $(t_{0}, \ldots, t_{n-1})$ are tuples from $I$ with 
$$
\text{qftp}_{L'}(s_{0}, \ldots, s_{n-1}) = \text{qftp}_{L'}(t_{0}, \ldots, t_{n-1}),
$$
then we have
$$
\text{tp}(a_{s_{0}},\ldots, a_{s_{n-1}}) = \text{tp}(a_{t_{0}},\ldots, a_{t_{n-1}}).
$$
\item We define the (generalized) \emph{EM-type} of $(a_{i})_{i \in I}$, written $\mathrm{EM}_{I}(a_{i} : i \in I)$, to be the partial type $\Gamma(x_{i} : i \in I)$ such that $\varphi(x_{i_{0}}, \ldots, x_{i_{n-1}}) \in \Gamma$ if and only if $\vDash \varphi(a_{j_{0}}, \ldots, a_{j_{n-1}})$ for all tuples $(j_{0}, \ldots, j_{n-1})$ from $I$ with $(j_{0}, \ldots, j_{n-1}) \vDash \mathrm{qftp}_{L'}(i_{0}, \ldots, i_{n-1})$. If $(b_{i} : i \in I) \vDash \mathrm{EM}_{I}(a_{i} : i \in I)$, we say $(b_{i} : i \in I)$ is \emph{locally based} on $(a_{i} : i \in I)$.  
\item We say that $I$-indexed indiscernibles have the \emph{modeling property} if, given any $(a_{i} : i \in I)$ from $\mathbb{M}$, there is an \(I\)-indexed indiscernible \((b_{i} : i \in I)\) in $\mathbb{M}$ locally based on $(a_{i} : i \in I)$. 
\end{enumerate}
\end{defn}

\begin{rem}
    When $I$-indexed indiscernibles have the modeling property and $J$ is an $L'$-structure with $\mathrm{Age}(I) = \mathrm{Age}(J)$, we additionally have that, given $(a_{i})_{i \in I}$, there is a $J$-indexed indiscernible $(b_{i})_{i \in J}$ locally based on $(a_{i})_{i \in I}$.  This follows easily by compactness, and we will often use the modeling property in this form. 
\end{rem}

For the remainder of the paper, except for the familiar case of indiscernible sequences, we will only ever consider $I$-indexed indiscernibles in the case where $I$ is a tree, though there are important differences between the notions of indiscernibility one obtains based on different choices of language for the tree $I$.  The language $L_{0}$ is the language consisting of two binary relations $\unlhd$ and $\leq_{lex}$, and a binary function $\wedge$.  The tree $\omega^{<\omega}$, for example, may be naturally viewed as an $L_{0}$-structure, where $\unlhd$ is interpreted the tree partial order, $\leq_{lex}$ as the lexicographic order, and $\wedge$ as the binary meet function.  If $I$ is an $L_{0}$-structure with $\mathrm{Age}(I) = \mathrm{Age}(\omega^{<\omega})$, then we refer to $I$-indexed indiscernibles as \emph{strongly indiscernible trees}.  

If $\alpha$ is an ordinal, we define a language $L_{s,\alpha}$ which consists of $L_{0}$, together with unary predicates $P_{\beta}$ for every $\beta < \alpha$. The tree $\omega^{<\alpha}$ can be viewed as an $L_{s,\alpha}$-structure by giving the symbols of $L_{0}$ their natural interpretation and interpreting each predicate $P_{\beta}$ as $\omega^{\beta}$, that is, as the set of nodes at level $\beta$ in the tree.  If $(a_{i})_{i \in I}$ is an $I$-indexed indiscernible for some $L_{s,\alpha}$-structure $I$ with $\mathrm{Age}(I) = \mathrm{Age}(\omega^{<\alpha})$ for some $\alpha$, then we refer to $(a_{i})_{i \in I}$ as an $s$\emph{-indiscernible tree}.

\begin{fact}\cite[Theorem 4.3]{KimKimScow} \cite[Theorem 16]{TakeuchiTsuboi} \label{modeling}
Let denote \(I_{s}\) be the \(L_{s,\omega}\)-structure \((\omega^{<\omega}, \unlhd, <_{lex}, \wedge, (P_{\alpha})_{\alpha < \omega})\) with all symbols being given their intended interpretations and each \(P_{\alpha}\) naming the elements of the tree at level \(\alpha\) and let $I_{0}$ denote its reduct to $L_{0} = \{ \unlhd, \leq_{lex}, \wedge \}$.  Then both $I_{0}$-indexed indiscernibles (strongly indiscernible trees) and \(I_{s}\)-indexed indiscernibles ($s$-indiscernible trees) have the modeling property. 
\end{fact}

\begin{rem} \label{two index models}
Trees of height greater than $\omega$ may also be considered as $s$-indiscernible trees, though this requires adding additional predicates to the language on the index model:  we say, for example, that $(a_{\eta})_{\eta \in \omega^{<\beta}}$ is an $s$-indiscernible tree if it is an $\omega^{<\beta}$-indexed indiscernible where $\omega^{<\beta}$ is considered as a structure in the language $L_{s,\beta}$ which contains predicates $(P_{\alpha})_{\alpha < \beta}$ for all $\beta$ levels of the tree.  As the language on the index model of an $s$-indiscernible tree is typically clear from context, we will not specify it explicitly.  
\end{rem}

We will use the phrase \emph{Fra\"iss\'e class} to denote a uniformly locally finite class of finite structures satisfying the hereditary property, the joint embedding property, and the amalgamation property. Given any $L$-structures $A,B$, we write $\mathrm{Emb}_{L}(A,B)$ to denote the set of embeddings from $A$ to $B$.  We omit the $L$ subscript when it is understood from context.

Recall that a Fra\"iss\'e class $\mathcal{K}$ is has the \emph{Ramsey property} if, given any $A \subseteq B$ and $r \in \omega$, there is some $C \in \mathcal{K}$ such that, if $\chi : \mathrm{Emb}(A,C) \to r$, there is some $\alpha \in \mathrm{Emb}(B,C)$ such that $\chi|_{\alpha \circ \mathrm{Emb}(A,B)}$ is constant, where 
$$
\alpha \circ \mathrm{Emb}(A,B) = \{\alpha \circ \beta : \beta \in \mathrm{Emb}(A,B)\}. 
$$
A Fra\"iss\'e class satisfying the Ramsey property is called a \emph{Ramsey class}.  

There is a tight connection between Ramsey classes and generalized indiscernibles with the modeling property, established by the following theorem of Scow:

\begin{fact} \cite[Theorem 3.12]{SCOW} \label{scows theorem}
Suppose $I$ is an infinite, locally finite structure expanding a linear order in the language $L'$, such that quantifier-free types are isolated by quantifier-free formulas. Then $I$-indexed indiscernibles have the modeling property if and only if $\mathrm{Age}(I)$ is a Ramsey class. 
\end{fact}

The language $L_{0,P} = \{\unlhd, \wedge, <_{lex}, P\}$ where $P$ is a unary predicate.  The class $\mathbb{K}_{0,P}$ consists of all finite $\wedge$-trees $A$ in which every element of $P^{A}$ is a leaf\textemdash that is, each $A \in \mathbb{K}_{0,P}$ satisfies the axiom
$$
(\forall \eta \in P)(\forall \nu)[\neg (\eta \vartriangleleft \nu)].  
$$
Note that if $\omega^{\leq \omega}$ is viewed as an $L_{0,P}$ structure in which $\wedge, \unlhd$, and $<_{lex}$ receives their natural interpretations and $P$ is interpreted as $\omega^{\omega}$, then $\mathrm{Age}(\omega^{\leq \omega}) = \mathbb{K}_{0,P}$. 

\begin{defn}
We define a \emph{treetop indiscernible} to be any $I$-indexed indiscernible where $I$ is an $L_{0,P}$-structure with $\mathrm{Age}(I) = \mathbb{K}_{0,P}$.  
\end{defn}

We aim to show that treetop indiscernibles have the modeling property or, equivalently, that $\mathbb{K}_{0,P}$ is a Ramsey class.  In the arguments below, it will be useful to introduce the following notation:  if $I$ is an $L_{0,P}$-structure with $\mathrm{Age}(I) = \mathbb{K}_{0,P}$, we will write $I_{+}$ for $P(I)$, and we will write $I_{-}$ for $I \setminus P(I)$.  In other words, $I_{+}$ names the leaves of the tree $I$ and $I_{-}$ names the non-leaves. 

Recall that the tree $\omega^{\leq \omega}$ may be viewed as an index model for $s$\emph{-indiscernible trees}, in which case this tree is viewed as a structure in the language $L_{s,\omega+1} = \{\wedge, \unlhd, \leq_{lex}, (P_{\alpha})_{\alpha \leq \omega}\}$, where $P_{\alpha}$ is interpreted as the $\alpha$th level of the tree.  We may regard the $L_{0,P}$-structure on $\omega^{\leq \omega}$ as a reduct of its $L_{s,\omega+1}$-structure, identifying $P$ with $P_{\omega}$.  

\begin{lem} \label{same type}
Suppose $\overline{\eta}, \overline{\nu}$ are $\wedge$-closed tuples from $\omega^{\leq \omega}$ and we write
\begin{eqnarray*}
\overline{\eta} &=& (\overline{\eta}_{-}, \overline{\eta}_{+}) \\
\overline{\nu} &=& (\overline{\nu}_{-}, \overline{\nu}_{+}),
\end{eqnarray*}
such that $\overline{\eta}_{-},\overline{\nu}_{-}$ are tuples from $\omega^{<\omega}$ and $\overline{\eta}_{+}, \overline{\nu}_{+}$ are from $\omega^{\omega}$.  Then if $\overline{\eta}_{-} = \overline{\nu}_{-}$ and $\mathrm{qftp}_{L_{0,P}}(\overline{\eta}) = \mathrm{qftp}_{L_{0,P}}(\overline{\nu})$, then we have $\mathrm{qftp}_{L_{s,\omega+1}}(\overline{\eta}) = \mathrm{qftp}_{L_{s,\omega+1}}(\overline{\nu})$.  
\end{lem}

\begin{proof}
Since $\overline{\eta}$ and $\overline{\nu}$ are $\wedge$-closed and $\mathrm{qftp}_{L_{0,P}}(\overline{\eta}) = \mathrm{qftp}_{L_{0,P}}(\overline{\nu})$, it is enough to show that the map $\overline{\eta} \mapsto \overline{\nu}$ preserves every predicate of the form $P_{i}$ for $i \leq \omega$.  But this mapping takes $\overline{\eta}_{-}$ to $\overline{\nu}_{-}$ so preserves $P_{i}$ for every $i < \omega$.  The mapping also takes $\overline{\eta}_{+}$ to $\overline{\nu}_{+}$, so preserves $P_{\omega}$ as well.  
\end{proof}

We will argue that $\mathrm{Age}_{L_{0,P}}(\omega^{\leq \omega})$ is a Ramsey class.  In order to do this, it suffices, by \cref{scows theorem}, to show the following:

\begin{lem} \label{coloring lemma}
Given any $(a_{\eta})_{\eta \in \omega^{\leq \omega}}$, there is some $(b_{\eta})_{\eta \in \omega^{\leq \omega}}$ which is treetop indiscernible and locally based on $(a_{\eta})_{\eta \in \omega^{\leq \omega}}$.  
\end{lem}

\begin{proof}
Let $(a'_{\eta})_{\eta \in \omega^{\leq \omega}}$ be an $s$-indiscernible tree locally based on $(a_{\eta})_{\eta \in \omega^{\leq \omega}}$.  

\begin{claim}\label{cla:reduction to s-ind}
It suffices to find $(b_{\eta})_{\eta \in \omega^{\leq \omega}}$ which is treetop indiscernible and locally based on $(a'_{\eta})_{\eta \in \omega^{\leq \omega}}$.
\end{claim}

\begin{proof}[Proof of Claim]
   Suppose $(b_{\eta})_{\eta \in \omega^{\leq \omega}}$ is treetop indiscernible and locally based on $(a'_{\eta})_{\eta \in \omega^{\leq \omega}}$.  Suppose further that $\overline{\eta}$ is a tuple from $\omega^{\leq \omega}$ and $\vDash \varphi(b_{\overline{\eta}})$.  By the local basedness of $(b_{\eta})_{\eta \in \omega^{\leq \omega}}$ as a treetop indiscernible, there is $\overline{\nu}$ in $\omega^{\leq \omega}$ with $\mathrm{qftp}_{L_{0,P}}(\overline{\eta}) = \mathrm{qftp}_{L_{0,P}}(\overline{\nu})$ and $\vDash \varphi(a'_{\overline{\nu}})$.  Then as $(a'_{\eta})_{\eta \in \omega^{\leq \omega}}$ is locally based on $(a_{\eta})_{\eta \in \omega^{\leq \omega}}$ as an $s$-indiscernible tree, there is $\overline{\xi}$ in $\omega^{\leq \omega}$ such that $\mathrm{qftp}_{L_{s,\omega+1}}(\overline{\nu}) = \mathrm{qftp}_{L_{s,\omega+1}}(\overline{\xi})$ and $\vDash \varphi(a_{\overline{\xi}})$.  It follows then that $\mathrm{qftp}_{L_{0,P}}(\overline{\eta}) = \mathrm{qftp}_{L_{0,P}}(\overline{\xi})$.  This shows $(b_{\eta})_{\eta \in \omega^{\leq \omega}}$ is locally based on $(a_{\eta})_{\eta \in \omega^{\leq \omega}}$.
\end{proof}

So now let $\mathrm{EM}_{L_{0,P}}((a'_{\eta})_{\eta \in \omega^{\leq \omega}})$ denote the partial type in the variables $(x_{\eta})_{\eta \in \omega^{\leq \omega}}$ consisting of the following set of formulas:
$$
\{\varphi(x_{\overline{\eta}}) : \mathbb{M} \vDash \varphi(a'_{\overline{\nu}}) \text{ for all }\overline{\nu} \vDash \mathrm{qftp}_{L_{0,P}}(\overline{\eta})\}.  
$$
Let $\Gamma$ denote the partial type consisting of $\mathrm{EM}_{L_{0,P}}((a'_{\eta})_{\eta \in \omega^{\leq \omega}})$ and the collection of formulas asserting that $(x_{\eta})_{\eta \in \omega^{\leq \omega}}$ is treetop indiscernible.  By \cref{cla:reduction to s-ind}, it suffices to show $\Gamma$ is consistent.  A finite subset of $\Gamma$ will be contained in 
$$
\mathrm{EMtp}_{L_{0,P}}((a'_{\eta})_{\eta \in \omega^{\leq \omega}})|_{x_{\overline{\xi}}} \cup \left\{ \mathrm{tp}_{\Delta}(x_{\overline{\eta}_{i}}) = \mathrm{tp}_{\Delta}(x_{\overline{\nu}_{i}}) : i < k\right\}
$$
for some finite $\Delta$, a finite tuple $\overline{\xi}$ from $\omega^{\leq \omega}$, and $\wedge$-closed tuples $\overline{\eta}_{i},\overline{\nu}_{i}$ with $\overline{\nu}_{i} \vDash \mathrm{qftp}_{L_{0,P}}(\overline{\eta}_{i})$ for all $i < k$.  Let $C$ be a finite $L_{0,P}$-substructure of $\omega^{\leq \omega}$ containing $\overline{\xi}$ and $\overline{\eta}_{i}, \overline{\nu}_{i}$ for all $i < k$ and so $C_{-}$ is the $L_{0}$-substructure of $\omega^{<\omega}$ consisting of the elements of $C \setminus P(C)$.  

For each $i < k$, let $q_{i} = \mathrm{qftp}_{L_{0,P}}(\overline{\eta}_{i})$ and define a coloring $c_{i} : q_{i}(\omega^{\leq \omega}) \to S^{l(\overline{\eta}_{i})}_{\Delta}(\emptyset)$ by 
$$
c_{i}(\overline{\zeta}) = \mathrm{tp}_{\Delta}(a'_{\overline{\zeta}})
$$
for all $\overline{\zeta} \in q_{i}(\omega^{\leq \omega})$.  Note that, since $\Delta$ is finite, we know $S^{l(\overline{\eta}_{i})}_{\Delta}(\emptyset)$ is finite.  

Let, for each $i < k$, $\overline{\eta}_{-,i}$ be the subtuple of $\overline{\eta}_{i}$ consisting of those elements not in $\omega^{\omega}$ and likewise for $\overline{\nu}_{-,i}$.  Let $q_{-,i} = \mathrm{qftp}_{L_{0}}(\overline{\eta}_{-,i}) = \mathrm{qftp}_{L_{0}}(\overline{\nu}_{-,i})$.  Then we define a coloring $c_{-,i} : q_{-,i}(\omega^{<\omega}) \to S^{l(\overline{\eta}_{i})}_{\Delta}(\emptyset)$ by setting, for each $\overline{\mu} \in q_{-,i}(\omega^{<\omega})$, 
$$
c_{-,i}(\overline{\mu}) = c_{i}(\overline{\zeta}) = \mathrm{tp}_{\Delta}(a'_{\overline{\zeta}})
$$
for any $\overline{\zeta} \in q_{i}(\omega^{\leq \omega})$ with $\overline{\zeta}_{-} = \overline{\mu}$.  By \cref{same type} and the $s$-indiscernibility of $(a'_{\eta})_{\eta \in \omega^{\leq \omega}}$, $c_{-,i}$ is well-defined.  As $\mathrm{Age}_{L_{0}}(\omega^{<\omega})$ is a Ramsey class, by \cref{modeling}, there is some $C_{-}' \cong C_{-}$, an $L_{0}$-substructure of $\omega^{<\omega}$, such that $c_{-,i}|_{q_{-,i}(C'_{-})}$ is constant for all $i < k$.  Choose any $C' \supseteq C'_{-}$, with $C'$ a substructure of $\omega^{\leq \omega}$ and $C'$ isomorphic to $C$ as an $L_{0,P}$-structure.  Then, unravelling definitions, we have that $c_{i}|_{q_{i}(C')}$ is constant for all $i < k$.  Letting $\overline{\xi}'$, $\overline{\eta}'_{i}$ and $\overline{\nu}'_{i}$ denote the corresponding tuples in $C'$, we have that $a'_{\overline{\xi'}}$, $(a'_{\overline{\eta}_{i}})_{i < k}$, and $(a'_{\overline{\nu}_{i}})_{i < k}$ realize the desired finite subset of $\Gamma$.  This concludes the proof.  
\end{proof}

\begin{cor} \label{Ramsey class cor}
$\mathbb{K}_{0,P}$ is a Ramsey class. 
\end{cor}

\begin{proof}
Immediate by \cref{coloring lemma} and \cref{scows theorem}. 
\end{proof}

As $\mathbb{K}_{0,P}$ is a Ramsey class, it is, in particular, a Fra\"iss\'e class, by \cite[Theorem 2.13]{bodirsky2015ramsey}.  We denote the Fra\"iss\'e limit of $\mathbb{K}_{0,P}$ by $\mathcal{T}$.  This structure will play an important role in the definition of treeless theories in the subsection below. 

\subsection{Treeless theories}

Given an $L_{0,P}$-structure $I$ with $\mathrm{Age}(I) = \mathrm{Age}(\omega^{\leq \omega})$ and $\eta \in I$, let $C(\eta) = \{\nu \in P(I) : \eta \unlhd \nu\}$, i.e. the leaves of $I$ that are in the cone above $\eta$. 

\begin{defn}
Say that $T$ is \emph{treeless} if whenever $(a_{\eta})_{\eta \in \mathcal{T}}$ is treetop indiscernible and $\xi \in \mathcal{T}$, then $(a_{\eta})_{\eta \in C(\xi)}$ is an indiscernible sequence over $a_{\xi}$ (i.e. is order-indiscernible over $a_{\xi}$ with respect to $<_{lex}$). 
\end{defn}

\begin{prop}
The following are equivalent: 
\begin{enumerate}
\item $T$ is treeless. 
\item If $\mathcal{S}$ is any $L_{0,P}$-structure with $\mathrm{Age}(\mathcal{S}) = \mathbb{K}_{0,P}$ and $(a_{\eta} : \eta \in \mathcal{S})$ is treetop indiscernible, then for any $\eta \in \mathcal{S}$, $(a_{\eta} : \eta \in C(\eta))$ is order indiscernible over $a_{\eta}$. 
\item If $(a_{\eta} : \eta \in \omega^{\leq \omega})$ is treetop indiscernible, then $(a_{\eta} : \eta \in \omega^{\omega})$ is order indiscernible over $a_{\emptyset}$. 
\end{enumerate}
\end{prop}

\begin{proof}
The implication $(2)\implies (1)$ is trivial and $(1) \implies (3)$ is easy, using that $\mathrm{Age}(\omega^{\leq \omega}) = \mathrm{Age}(\mathcal{T})$, so we show $(3)\implies (2)$.  Assume (3) and suppose $\mathcal{S}$ is an $L_{0,P}$-structure with $\mathrm{Age}(\mathcal{S})=\mathbb{K}_{0,P}$, $(a_{\eta} : \eta \in \mathcal{S})$ is a treetop indiscernible, and $\xi \in \mathcal{S}_{-}$.  We must show $(a_{\eta} : \eta \in C(\xi))$ is order-indiscernible over $a_{\xi}$. Note that the $L_{0,P}$-substructure $S_{\xi}$ consisting of all $\eta \in \mathcal{S}$ with $\xi \unlhd \eta$ satisfies $\mathrm{Age}(S_{\xi}) \supseteq \mathrm{Age}(\omega^{\leq \omega})$.  Consequently, for each finite tuple $\overline{\eta}$ from $\omega^{\leq \omega}$, there is some $\overline{\nu}$ in $\mathcal{S}_{\xi}$ such that $\mathrm{qftp}_{L_{0,P}}(\overline{\eta}) = \mathrm{qftp}_{L_{0,P}}(\overline{\nu})$.  We define the type $p_{\overline{\eta}}(x_{\overline{\eta}})$ to be $\text{tp}(a_{\overline{\nu}})$ for some (equivalently, all) such $\overline{\nu}$.  Then, by compactness, $\Gamma(x_{\eta} : \eta \in \omega^{\leq \omega}) = \bigcup_{\overline{\eta}} p_{\overline{\eta}}$ is consistent, where $\overline{\eta}$ ranges over all finite tuples of $\omega^{\leq \omega}$.  Moreover, letting $(b_{\eta} : \eta \in \omega^{\leq \omega})$ be a realization, we have that $(b_{\eta} : \eta \in \omega^{\leq \omega})$ is treetop indiscernible.  By assumption, then, $(b_{\eta} : \eta \in \omega^{\omega})$ is order indiscernible over $b_{\emptyset}$.  By construction, this entails that $(a_{\eta} : \eta \in C(\xi))$ is order indiscernible over $a_{\xi}$.  As the case of $\xi \in \mathcal{S}_{+}$ is trivial, this completes the proof. 
\end{proof}

If $T$ is NIP, the definition of \emph{treeless} can be weakened to omit the condition that the leaves are order indiscernible \emph{over the root}:

\begin{prop} \label{no root}
Assume $T$ is NIP.  Suppose that for all treetop indiscernibles $(a_{\eta})_{\eta \in \omega^{\leq \omega}}$, the sequence $(a_{\eta})_{\eta \in \omega^{\omega}}$ is an indiscernible sequence. Then $T$ is treeless. 
\end{prop}

\begin{proof}
Suppose $(a_{\eta})_{\eta \in \omega^{\leq \omega}}$ is treetop indiscernible.  We must show that $(a_{\eta})_{\eta \in \omega^{\omega}}$ is indiscernible over $a_{\emptyset}$.  By compactness, we may stretch the given treetop indiscernible to $(a_{\eta})_{\eta \in \kappa^{\leq \omega}}$ with $\kappa = |T|^{+}$.  Since $T$ is NIP, by \cite[Proposition 2.8]{simon2015guide}, there is an end segment $J \subseteq \kappa^{\omega}$ such that $(a_{\eta})_{\eta \in J}$ is $a_{\emptyset}$-indiscernible. By treetop indiscernibility, it follows that $(a_{\eta})_{\eta \in \kappa^{\omega}}$ is $a_{\emptyset}$-indiscernible as well.  Therefore $T$ is treeless. 
\end{proof}

\begin{quest}
Is \cref{no root} true without the assumption that $T$ is NIP? Note that weakened notion of treeless, in which the leaves indexed by $\omega^{\omega}$ in a treetop indiscernible $(a_{\eta})_{\eta \in \omega^{\leq \omega}}$ are only required to be an indiscernible sequence (not necessarily indiscernible over $a_{\emptyset}$) suffices for many of the observations.
\end{quest}

The following related question was suggested to us by Artem Chernikov:

\begin{quest}
    To check treelessness, does it suffice to consider triples of leaves?  More precisely, if whenever $(a_{\eta})_{\eta \in \mathcal{T}}$ is a treetop indiscernible and, for all $\eta_{0} <_{lex} \eta_{1} <_{lex} \eta_{2}$ and $\nu_{0} <_{lex} \nu_{1} <_{lex} \nu_{2}$ from $\mathcal{T}^{+}$, we have $(a_{\eta_{0}},a_{\eta_{1}}, a_{\eta_{2}}) \equiv_{a_{\emptyset}} (a_{\nu_{0}}, a_{\nu_{1}}, a_{\nu_{2}})$, does it follow that $T$ is treeless?  
\end{quest}

\begin{exmp}
Any structure homogeneous in a binary language.  Any theory of a pure linear order is (distal and) treeless, since it eliminates quantifiers in a binary language \cite[Lemma A.1]{simon2015guide}.
\end{exmp}

\begin{exmp}
The theory of any ordered abelian group is not treeless. To see this, let $G$ be any ordered abelian group. We may assume $G$ is $\aleph_{0}$-saturated and hence we can fix some $g > 0$ in $G$ which is $n$-divisible for all $n$ (take $g$ to be in the intersection of $n\cdot G$ for all $n<\omega$).  Fix $2 \leq n,m < \omega$.  Then for each $\eta \in n^{\leq m}$, as $g$ is $k$-divisible for all $k$, we can define 
$$
a_{\eta} = \sum_{i < m} \frac{\eta(i)}{n^{i}}g \in G.
$$
Consider some $\eta_{0} <_{lex} \eta_{1} <_{lex} \eta_{2} <_{lex} \eta_{3}$ in $n^{m}$ with
$$
(\eta_{0} \wedge \eta_{1}) \vartriangleright \eta_{1} \wedge (\eta_{0} \wedge \eta_{2})
$$
and 
$$
(\eta_{2} \wedge \eta_{3}) \vartriangleright (\eta_{0} \wedge \eta_{2})
$$
(and thus $(\eta_{0} \wedge \eta_{2}) = (\eta_{0} \wedge \eta_{3}) = (\eta_{1} \wedge \eta_{2}) = (\eta_{1} \wedge \eta_{3})$).  Then we have
$$
a_{\eta_{1}} - a_{\eta_{0}} < a_{\eta_{3}} - a_{\eta_{1}}
$$
and 
$$
a_{\eta_{2}} - a_{\eta_{0}} > a_{\eta_{3}} - a_{\eta_{2}}.  
$$
Hence, by compactness and \cref{Ramsey class cor}, we can find a treetop indiscernible $(b_{\eta})_{\eta \in \omega^{\leq \omega}}$ in a model of $\mathrm{Th}(G)$ satisfying the same pair of inequalities, which shows that $(b_{\eta})_{\eta \in \omega^{\omega}}$ is not an indiscernible sequence, hence $\mathrm{Th}(G)$ is not treeless. 
\end{exmp}

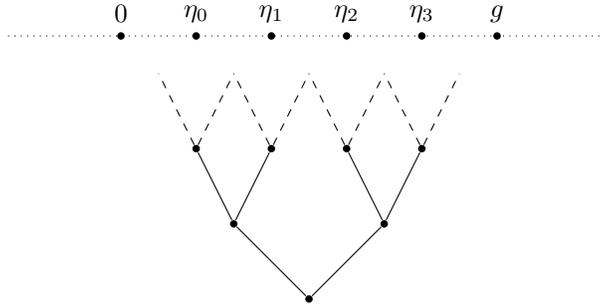
\begin{center}
\begin{tikzpicture}[level distance=1cm,                     level 1/.style={sibling distance=2cm},                     level 2/.style={sibling distance=1cm}]
  
  \draw[dotted] (-4,3.5) -- (4,3.5);
  \node[fill,circle,inner sep=1pt,label=above:{$0$}] at (-2.5,3.5) {};
  \node[fill,circle,inner sep=1pt,label=above:{$\eta_{0}$}] at (-1.5,3.5) {};
  \node[fill,circle,inner sep=1pt,label=above:{$\eta_{1}$}] at (-0.5,3.5) {};
  \node[fill,circle,inner sep=1pt,label=above:{$\eta_{2}$}] at (0.5,3.5) {};
  \node[fill,circle,inner sep=1pt,label=above:{$\eta_{3}$}] at (1.5,3.5) {};
  \node[fill,circle,inner sep=1pt,label=above:{$g$}] at (2.5,3.5) {};

  \node[circle,fill,inner sep=1pt] (root) {}[grow=up]
    child[solid] {node[circle,fill,inner sep=1pt] {}
      child[solid] {node[circle,fill,inner sep=1pt] {}
        child[dashed] {}
        child[dashed] {}
      }
      child[solid] {node[circle,fill,inner sep=1pt] {}
        child[dashed] {}
        child[dashed] {}
      }
    }
    child[solid] {node[circle,fill,inner sep=1pt] {}
      child[solid] {node[circle,fill,inner sep=1pt] {}
        child[dashed] {}
        child[dashed] {}
      }
      child[solid]  { node[circle,fill,inner sep=1pt] {}
            child[dashed] {}
            child[dashed] {}     
      }
    };
\end{tikzpicture}

\end{center}

\begin{rem}
Even if $T$ is treeless, it may be the case that $(a_{\eta})_{\eta \in \omega^{\leq \omega}}$ is $s$-indiscernible and $(a_{\eta})_{\eta \in \omega^{\omega}}$ is not an indiscernible sequence (this $(a_{\eta})_{\eta \in \omega^{\leq \omega}}$ will be necessarily not treetop indiscernible).  For example, let $T$ be the model companion of the theory in the language $L = \{R_{n} : n < \omega\}$ that says that the binary relation $R_{n}$ is a graph for each $n$.  So in $T$, each $R_n$ defines a random graph and these graphs interact totally independently.  We may choose vertices $(a_{\eta})_{\eta \in \omega^{\leq \omega}}$ so that, for leaves $\eta, \nu \in \omega^{\omega}$, $\vDash R_{n}(a_{\eta}, a_{\nu})$ holds if and only if the length of $\eta \wedge \nu$ is $n$.  This is preserved when passing to an s-indiscernible tree locally based on the $(a_{\eta})_{\eta \in \omega^{\leq \omega}}$, so we can assume $(a_{\eta})_{\eta \in \omega^{\leq \omega}}$ is $s$-indiscernible.  Clearly $(a_{\eta})_{\eta \in \omega^{\omega}}$ is not an indiscernible sequence.  However, $T$ eliminates quantifiers and the language $L$ is binary, so $T$ is treeless. 
\end{rem}

\begin{prop} \label{interpretation}
Suppose the theory $T'$ is interpretable in the treeless theory $T$.  Then $T'$ is treeless.  
\end{prop}

\begin{proof}
Suppose $T'$ is interpretable in $T$ and $E$ is a $T$-definable equivalence relation such that if $M \vDash T$, then $M^{n}/E$ is the domain of a model of $T'$ whose relations are definable in $T$. Let $\mathbb{M}' = \mathbb{M}^{n}/E$ and let $\pi : \mathbb{M}^{n} \to \mathbb{M}'$ denote the interpretation map.  Suppose $(a_{\eta})_{\eta \in \omega^{\leq \omega}}$ is a treetop indiscernible in $\mathbb{M}'$.  Then for each $\eta \in \omega^{\leq \omega}$, we can choose some $\tilde{a}_{\eta} \in \pi^{-1}(a_{\eta})$.  We can then take $(b_{\eta})_{\eta \in \omega^{\leq \omega}}$ which is treetop indiscernible and locally based on $(\tilde{a}_{\eta})_{\eta \in \omega^{\leq \omega}}$ in $\mathbb{M}$.  As $T$ is treeless, $(b_{\eta})_{\eta \in \omega^{\omega}}$ is an indiscernible sequence over $b_{\emptyset}$.  In particular, $(\pi(b_{\eta}))_{\eta \in \omega^{\omega}}$ is an indiscernible sequence over $\pi(b_{\emptyset})$.  But since $(a_{\eta})_{\eta \in \omega^{\leq \omega}}$ was taken to be treetop indiscernible in $\mathbb{M}'$, we have, by local basedness, that $(a_{\eta})_{\eta \in \omega^{\leq \omega}} \equiv (\pi(b_{\eta}))_{\eta \in \omega^{\leq \omega}}$, hence $(a_{\eta})_{\eta \in \omega^{\omega}}$ is an indiscernible sequence over $a_{\emptyset}$, which shows $T'$ is treeless. 
\end{proof}

Recall the following:

\begin{defn}
Suppose $k \geq 1$.  We say that a formula $\varphi(x;y_{0}, \ldots, y_{k-1})$ has the $k$\emph{-independence property} ($k$\emph{-IP}) if there is some array $(a_{i,j} : i < k, j < \omega)$ such that, for all $X \subseteq \omega^{k}$, there is some $b_{X}$ such that 
$$
\vDash \varphi(b_{X}, a_{0,j_{0}}, a_{1,j_{1}}, \ldots, a_{k-1,j_{k-1}}) \iff (j_{0}, \ldots, j_{k-1}) \in X. 
$$
We say that a theory $T$ has the $k$-independence property if some formula does modulo $T$. A theory without $k$-IP is called $k$\emph{-dependent}. 
\end{defn}

Note that if a theory is $k$-dependent, then it is $k'$-dependent for all $k' \geq k$. The independence property is the same as $1$-IP.  The $k$-dependence hierarchy was introduced by Shelah in \cite{shelah2007definable}.  See also \cite{MR3952231} for further details on these classes of theories. 

\begin{prop}
If $T$ is treeless, then $T$ is $2$-dependent. In particular, $T$ is $k$-dependent for all $k \geq 2$. 
\end{prop}

\begin{proof}
We prove the contrapositive.  Suppose $T$ has $2$-IP witnessed by the formula $\varphi(x;y,z)$.  Then, by compactness, there is a sequence $(b_{\eta}, c_{\eta} : \eta \in \omega^{\omega})$ such that, for all $X \subseteq \omega^{\omega} \times \omega^{\omega} $, there is some $a_{X}$ such that 
$$
\vDash \varphi(a_{X};b_{\eta}, c_{\nu}) \iff (\eta, \nu) \in X. 
$$
Now for each $\eta \in \omega^{\omega}$, let 
$$
X_{\eta} = \{(\nu, \xi) \in \omega^{\omega} \times \omega^{\omega} : \eta <_{lex} \nu <_{lex} \xi \text{ and } \eta \wedge \nu \vartriangleleft \nu \wedge \xi \}. 
$$
Choose, for each $\eta \in \omega^{\omega}$ some $a_{\eta}$ such that 
$$
\vDash \varphi(a_{\eta}; b_{\nu}, c_{\xi}) \iff (\nu, \xi) \in X_{\eta}. 
$$
Choose a sequence of same-length tuples $(d_{\eta})_{\eta \in \omega^{<\omega}}$ arbitrarily and set $d_{\eta} = (a_{\eta}, b_{\eta}, c_{\eta})$ for each $\eta \in \omega^{\omega}$.  Let $(d'_{\eta})_{\eta \in \omega^{\leq \omega}}$ be a treetop indiscernible locally based on $(d_{\eta})_{\eta \in \omega^{\leq \omega}}$ and write $d'_{\eta} = (a'_{\eta}, b'_{\eta}, c'_{\eta})$ for each $\eta \in \omega^{\omega}$. Note that we still have 
$$
\vDash \varphi(a'_{\eta}, b'_{\nu},c'_{\xi}) \iff \eta <_{lex} \nu <_{lex} \xi \text{ and } \eta \wedge \xi \vartriangleleft \nu \wedge \xi. 
$$
Choosing $\eta_{0} <_{lex} \eta_{1} <_{lex} \eta_{2} <_{lex} \eta_{3}$ in $\omega^{\omega}$ with $\eta_{0} \wedge \eta_{1} \vartriangleright \eta_{1} \wedge \eta_{3}$ and $\eta_{0} \wedge \eta_{2} \vartriangleleft \eta_{2} \wedge \eta_{3}$, we have $\vDash \neg \varphi(a'_{\eta_{0}}, b'_{\eta_{1}}, c'_{\eta_{3}})$ and $\vDash \varphi(a'_{\eta_{0}}, b'_{\eta_{2}}, c'_{\eta_{3}})$, so $(d'_{\eta})_{\eta \in \omega^{\omega}}$ is not order-indiscernible. 
\end{proof}

\section{Symmetry and base monotonicity in treeless theories} \label{sec:symmetry and base mon}

In this section we will prove that GS-independence enjoys symmetry and base monotonicity in treeless theories. To do that we start by introducing a generalization of the product operation discussed below \cref{lem_def}.

\subsection{A generalization of the product operator}

\begin{defn}
  Let $\pi(x)$ is a global partial type which is ind-definable over $Ac$ where $c$ is a $y$-tuple. For any $b\equiv_A c$, let $\pi(x,b)$ be the type we get after applying an automorphism fixing $A$ mapping $c$ to $b$. In other words, 
  if $\pi(x)$ is defined by the collection of pairs $(\phi_{i}(x;z),d\phi_{i}(z,c))$ 
  where $\phi_i \in L$ and $d\phi \in L(A)$ then $\phi(x,b)$ is defined by 
  $(\phi_{i}(x;z),d\phi_{i}(z,b))$.
\end{defn}

\begin{rem} \label{rem:not what you think}
  We note that $\pi(x,b)$ is \emph{not} obtained by simply replacing instances of $c$ in $\pi$ with $b$.  Consider, for example, the theory $T$ of an equivalence relation with infinitely many classes, all of which are infinite.  Let $A = \emptyset$ and $c$ any element and consider $\pi(x)$ the global non-forking extension over $c$ of the type axiomatized by $\{E(x,c)\}$.  As $T$ is stable, $\pi(x)$ is generically stable over $c$. Let $d \neq c$ be some element in the same class as $c$ and let $b$ be an element in a different class.  Then $E(x,d) \wedge E(x,c) \in \pi(x)$.  Simply replacing $c$ with $b$ would produce $E(x,d) \wedge E(x,b)$ which is inconsistent. In this situation, $\pi(x,b)$ is the global non-forking extension of the type over $b$ axiomatized by $E(x,b)$. 
  \end{rem}

\begin{lem}\label{lem:twisted product}
  Suppose $\pi(x)$ is a global partial type which is ind-definable over $Ac$ where $c$ is a $y$-tuple and that $\lambda(y) \supseteq \tp(c/A)$ is an $A$-ind-definable global partial type. Then there is a unique ind-definable over $A$ partial type $(\pi \rtimes \lambda) (x,y)$ such that for any $B\supseteq A$, $(a,b) \vDash (\pi \rtimes \lambda) (x,y)|_B$ if and only if $b\vDash \lambda|_B$ and $a \vDash \pi(x,b)|_{Bb}$.

  It follows that if $\pi(x)$ is ind-definable over $A$ then $\pi \rtimes \lambda = \pi \otimes \lambda$. 
\end{lem}

\begin{proof}
  Let $\phi(x,y,z)$ be a formula in $L$ (without parameters). Let $S^\pi_\phi$ be the collection of formulas $\psi(y,z) \in L(A)$ such that for all $d$, $\phi(x,c,d) \in \pi$ if and only if $\psi(c,d)$ holds for some $\psi \in S_\phi$. Note that (*) for any $b\equiv_A c$, $S^{\pi(x,b)}_\phi = S^\pi_\phi$, so we can discard the $\pi$ in the notation and write $S_\phi$.
  Let $(\pi \rtimes \lambda) (x,y)$ be the closure under finite conjunctions and logical consequences of $\lambda(y) \cup \{\psi(y,d) \to \phi(x,y,d) : \phi(x,y,z)\in L,\psi(y,z) \in S_{\phi}\}$. Note that $\pi \rtimes \lambda$ is ind-definable over $A$ as $\lambda(y)$ is and the second part is ind-definable by the defining scheme $(\psi(y;z) \to \phi(x,y;z);z=z)$ where $\psi \in S_\phi$. It clearly satisfies the requirement by (*) above. 

  Uniqueness follows by the fact a global partial type is determined by the realizations of its restrictions to small sets. In the case when $\pi$ is ind-definable over $A$, note that $\pi(x,b) = \pi$ for any $b\equiv_A c$, so that uniqueness implies that $\pi \rtimes \lambda = \pi \otimes \lambda$.
\end{proof}

\begin{rem}
  In the context of \cref{rem:not what you think}, letting $\lambda(y)=\tp(c)$, $(\pi \rtimes \lambda) (x,y)$ is axiomatized by $\{E(x,y)\}$.
\end{rem}

\begin{prop} \label{prop:alg semidirect product}
    Let $T$ be any theory. Let $b \in \acl(A)$, $\pi(x)$ be generically stable over $Ab$ and let $\lambda(y)= \tp(b/A)$. Then $(\pi \rtimes \lambda) (x,y)$ is generically stable over $A$.  
\end{prop}

\begin{proof}
    $\pi \rtimes \lambda$ is ind-definable over $A$ by \cref{lem:twisted product}.

    We show that $\pi \rtimes \lambda$ is generically stable over $A$. Let $(a_i,b_i:i<\omega)$ be a Morley sequence in $\pi \rtimes \lambda$ over $A$. Assume for a contradiction that there is $d$ and a formula $\phi(x,y;z)\in L(A)$ so that $\neg \phi(x,y;d)\in \pi \rtimes \lambda$ and $\bigwedge_{i<\omega} \phi(a_i,b_i;d)$ holds. We may assume that $(a_i,b_i:i<\omega)$ is $Ad$-indiscernible. 

    As $b \in \acl(A)$ and $(b_i : i<\omega)$ is $A$-indiscernible in the type of $b$ over $A$, there is some $b'$ such that $b_i=b'$ for all $i<\omega$. Thus, we have that $b' \vDash \lambda|_A$ (trivially) and $a_i \vDash \pi(x,b')|_{Aa_{<i}b'}$ for all $i<\omega$. Additionally, $(a_i : i<\omega)$ is indiscernible over $Adb'$ so by generic stability $a_i \vDash \pi(x,b')|_{Aa_{<i}b'd}$, contradiction. 
\end{proof}

The following proposition is a strengthening of \cref{prop:alg semidirect product} to any generically stable type $\lambda$ provided $T$ is treeless.

\begin{prop} \label{useful prop}
Assume that $T$ is treeless. Let $\pi(x)$ be generically stable over $Ac$ and let $\lambda(y)\supseteq \mathrm{tp}(c/A)$ be generically stable over $A$. Then $(\pi \rtimes \lambda) (x,y)$ is generically stable over $A$.  
\end{prop}
\begin{proof}
The type $\pi \rtimes \lambda$ is ind-definable over $A$ by \cref{lem:twisted product}.

We show that $\pi \rtimes \lambda$ is generically stable over $A$. Let $(a_i,b_i:i<\omega)$ be a Morley sequence in $\pi \rtimes \lambda$ over $A$. Assume for a contradiction that there is $d$ and a formula $\phi(x,y;z)\in L(A)$ so that $\neg \phi(x,y;d)\in \pi \rtimes \lambda$ and $\bigwedge_{i<\omega} \phi(a_i,b_i;d)$ holds. We may assume that $(a_i,b_i:i<\omega)$ is $Ad$-indiscernible. By generic stability of $\lambda$, it follows that $(b_i:i<\omega)$ is a Morley sequence of $\lambda$ over $Ad$. 

We extend this sequence to a tree $(c_{\eta} : \eta \in \omega^{<\omega})$ so that:

\begin{itemize}
\item For $\eta\in \omega^{<\omega}$,
$$
c_{\eta} = (\overline{a}_{\eta}, b_{\eta}) = ((a_{\eta,i} : i < \omega), b_{\eta}),
$$
where $(a_{\eta,i} : i < \omega)$ is a Morley sequence in $\pi(x;b_{\eta})$ over $Ab_{\eta}$.    
\item For every $\eta \in \omega^{\omega}$, the sequence $(a_{\eta | i, \eta(i)}, b_{\eta | i} : i < \omega)$ has the same type as $(a_{i}, b_{i} : i < \omega)$ over $A$.  
\end{itemize}

To build the tree, we start by taking a Morley sequence $(a'_{i} : i < \omega)$ in $\pi(x;b_{0})|_{Ab_{0}}$ with $a'_{0} = a_{0}$, and then we set $c_{\emptyset} = ((a'_{i} : i < \omega), b_{0})$. Assume we have constructed $(c_{\eta})_{\eta \in \omega^{< k}}$ such that, for all $\eta \in \omega^{k}$, 
$$
(a_{\eta | i, \eta(i)}, b_{\eta | i} : i < k) \equiv_{A} (a_{i}b_{i} : i < k).
$$
Fix $\eta \in \omega^{k}$.  Then we choose $a'$ and $b'$ so that 
$$
(a_{\eta | i, \eta(i)}, b_{\eta | i} : i < k)a'b' \equiv_{A} (a_{i}b_{i} : i < k)a_{k}b_{k}.  
$$
Since $a_{k} \vDash \pi(x;b_{k})|_{Aa_{<k}b_{\leq k}}$, we know $a'$ satisfies $\pi(x;b')$ restricted to $A(a_{\eta | i, \eta(i)}, b_{\eta | i} : i < k)b'$ and therefore
$$
\pi(x;b') \cup \text{tp}(a'/(a_{\eta | i, \eta(i)}, b_{\eta | i} : i < k)b')
$$
is consistent.  We then choose $\overline{a}' = (a'_{i} :  i< \omega)$ to be a Morley sequence in this type with $a'_{0} = a'$ and define $c_{\eta} = (\overline{a}_{\eta}, b_{\eta})$ by setting $\overline{a}_{\eta} = \overline{a}'$ and $b_{\eta} = b'$.  This defines $(c_{\eta} : \eta \in \omega^{< k+1})$ and thus, by induction, all of $(c_{\eta})_{\eta \in \omega^{<\omega}}$. 

For each branch $\eta\in \omega^{\omega}$, we can find $d_\eta\equiv_A d$ so that $\bigwedge_{n<\omega} \phi(a_{\eta|n, \eta(n)},b_{\eta|n};d_\eta)$ holds.  Define a tree $(e_{\eta} : \eta \in \omega^{\leq \omega})$ by setting $e_{\eta} = b_{\eta}$ for $\eta \in \omega^{<\omega}$ and $e_{\eta} = d_{\eta}$ for $\eta \in \omega^{\omega}$.  Let $(e'_\eta:\eta\in \omega^{\omega}) = (b'_{\eta}  :\eta \in \omega^{<\omega})^{\frown} (d'_{\eta} : \eta \in \omega^{\omega})$ be treetop indiscernible over $A$, locally based on $(e_{\eta} : \eta \in \omega^{\leq \omega})$. By compactness, we can stretch our treetop indiscernible to $(e'_{\eta} : \eta \in \kappa^{\leq \omega})$ for $\kappa = (|T|+|A|)^{+}$.  By treelessness, then, $(d'_{\eta}:\eta\in \kappa^{\omega})$ is order-indiscernible over $Ab'_{\emptyset}$.

By induction, we will build a path $\eta_{*} \in \kappa^{\omega}$ and sequences $\overline{a}_{n} = (a_{n,i} : i < \kappa)$ for each $n < \omega$ such that 
\begin{itemize}
\item For all $n < \omega$, $\overline{a}_{n}$ realizes $\pi^{(\kappa)}(x;b'_{\emptyset})$ over $Ab'_{\emptyset}$. 
\item $(a_{n,\eta_{*}(n)} : n < \omega)$ realizes $\pi^{(\omega)}(x;b'_{\emptyset})$ over $Ab'_{\emptyset}$ in $\pi(x;b'_{\emptyset})$.  
\item For all $n < \omega$, $\vDash \phi(a_{n,\eta_{*}(n)}, b'_{\emptyset}; d'_{\eta_{*}})$.  
\end{itemize}
For each $i < \kappa$ and $\nu \in \kappa^{<\omega}$, recall the notation $C(\nu^{\frown} \langle i\rangle) = \{\eta \in \kappa^{\omega} : \nu^{\frown}\langle i \rangle \unlhd \eta\}$.  As $(e'_{\eta}: \eta \in \kappa^{\leq \omega})$ is locally based on $(e_{\eta} : \eta \in \omega^{\leq \omega})$, we know
\[
(\exists \overline{x})\left[ \pi^{(\kappa)}(\overline{x};b'_{\emptyset})|_{Ab'_{\emptyset}} \wedge \bigwedge_{i < \kappa} \bigwedge_{\eta \in C(\langle i \rangle)} \phi(x_{i},b'_{\emptyset};d_{\eta})\right],
\]
since this follows from 
\[
(\exists \overline{x})\left[ \pi^{(\omega)}(\overline{x};b_{\nu})|_{Ab_{\nu}} \wedge \bigwedge_{i < \omega} \bigwedge_{\eta \in C(\nu^{\frown}\langle i\rangle)} \phi(x_{i},b_{\nu};d_{\eta})\right],
\]
and this was witnessed by $(a_{\nu,i} : i < \omega)$.  Therefore, we can let $\overline{a}_{0} = (a_{0,i} : i < \kappa)$ be a Morley sequence in $\pi(x;b'_{\emptyset})$ such that $\vDash \bigwedge_{i < \omega} \bigwedge_{\eta \in C(\langle i\rangle)} \phi(a_{0,i},b'_{\emptyset},d_{\eta})$.  We set $\eta_{*}(0) = 0$. 

Assume we have constructed $\eta_{*}|k$ for $k>0$.  Since $(d'_{\eta} : \eta \in \omega^{\omega})$ is an indiscernible sequence over $Ab'_{\emptyset}$, we know 
$$
\langle (d'_{\eta})_{\eta \in C(\langle i \rangle)}: i < \kappa \rangle \equiv_{Ab'_{\emptyset}} \langle (d'_{\eta})_{\eta \in C((\eta_{*}|k)^{\frown} \langle i\rangle)}: i < \kappa \rangle.  
$$
Choose $\overline{a}_{k}$ such that 
$$
\overline{a}_{0}\langle (d'_{\eta})_{\eta \in C(\langle i\rangle)}: i < \kappa \rangle \equiv_{Ab_{\emptyset}} \overline{a}_{k} \langle (d'_{\eta})_{\eta \in C((\eta_{*}|k)^{\frown}\langle i\rangle)}: i < \kappa \rangle.  
$$
Then $\overline{a}_{k}$ is a Morley sequence in $\pi(x;b'_{\emptyset})$ over $Ab'_{\emptyset}$ and we have 
$$
\vDash \phi(a_{k,i}, b'_{\emptyset}, d'_{\eta})
$$
for all $\eta \in \kappa^{\omega}$ with $\eta_{*}|k \unlhd \eta$ and $\eta(k) = i$.  By generic stability and the choice of $\kappa$, there is some $i_{*} < \kappa$ such that
$$
a_{k,i_{*}} \vDash \pi(x;b_{\emptyset})|_{A(a_{i,\eta_{*}(i)})_{i < k}}.
$$
Then we set $\eta_{*}(k) = i_{*}$.  

We have constructed a path $\eta_* \in \kappa^{\omega}$ so that $(a_{n, \eta_*(n)})_{n < \omega}$ realizes $\pi^{(\omega)}(x,b'_{\emptyset})$ over $Ab'_{\emptyset}$. Extracting, we may assume that this sequence is indiscernible over $Ad'_{\eta_*}$. Now $b'_{\emptyset}\vDash \lambda|_{Ad'_{\eta_*}}$ and since $\pi(x,b'_{\emptyset})$ is generically stable over $Ab'_{\emptyset}$, also $(a_{n,\eta_*(n)}:n<\omega)$ is a Morley sequence of $\pi(x,b'_{\emptyset})$ over $Ab'_{\emptyset}d'_{\eta_*}$. Hence $(a_{0, \eta_{*}(0)},b'_{\emptyset})\vDash \pi \rtimes \lambda |_{Ad'_{\eta_*}}$. Contradiction.
\end{proof}
\subsection{Symmetry and base monotonicity in treeless theories}

\begin{cor} \label{treeless implies P}
Assume $T$ is treeless.  Then property (P) holds. More generally, if $\pi(x)$ and $\lambda(y)$ are generically stable over $A$, then so is $(\pi\otimes \lambda)(x,y)$.
\end{cor}

\begin{proof}
Pick any $c \vDash \lambda|_{A}$ and let $\lambda' = \lambda \cup \text{tp}(c/A)$. Then by the ``it follows'' part of \cref{lem:twisted product}, $\pi \rtimes \lambda' = \pi \otimes \lambda'$. 
Thus, $(\pi \otimes \lambda')$ is generically stable by \cref{useful prop}. 

Suppose, towards contradiction, that $(\pi \otimes \lambda)$ is not generically stable over $A$.  Then there is a sequence $I = (a_{i},b_{i} : i < \omega) \vDash (\pi \otimes \lambda)^{(\omega)}|_{A}$ and some $\varphi(x,y;d) \in (\pi \otimes \lambda)$ such that $\bigwedge_{i < \omega} \neg \varphi(a_{i},b_{i};d)$.  After extracting, we may assume that $I$ is $Ad$-indiscernible.  Then for $\lambda' = \lambda \cup \text{tp}(b_{0}/A)$, we have $I \vDash (\pi \otimes \lambda')^{(\omega)}|_{A}$ and $\varphi(x,y;d) \in \pi \otimes \lambda'$, contradicting the generic stability of $(\pi \otimes \lambda')$. 
\end{proof}

\begin{cor} \label{sym+basemono}
If $T$ is treeless, then $\ind^{\mathrm{GS}}$ satisfies symmetry and base monotonicity.
\end{cor}

\begin{proof}
Symmetry follows by \cref{treeless implies P} and \cref{P implies sym}. To see base monotonicity, assume $a \ind^{\mathrm{GS}}_{A} bc$.  We want to show $a \ind^{\mathrm{GS}}_{Ac} b$.  If not, then $b \not\vDash \pi(x)|_{Aac}$, where $\pi(x)$ is the maximal global type extending $\text{tp}(b/Ac)$ which is generically stable over $Ac$. Let $\lambda(y) = \text{tp}(c/A)$.  Since $\lambda(y)$ is generically stable over $A$ and $\pi(x)$ is generically stable over $Ac$, \cref{useful prop} implies that $(\pi \rtimes \lambda) (x,y)$ is generically stable over $A$.  Since $(b,c) \vDash \pi \rtimes \lambda|_{A}$ if and only if $c \vDash \lambda|_{A}$ and $b \vDash \pi|_{Ac}$, we know that $(b,c) \not\vDash \pi \rtimes \lambda$.  However, $\pi \rtimes \lambda$ is consistent with $\text{tp}(bc/A)$, so this contradicts our assumption that $a \ind^{\mathrm{GS}}_{A} bc$.  
\end{proof}

\begin{cor} \label{rosy}
If $T$ is treeless, then $T$ is rosy.
\end{cor}

\begin{proof}
By \cite[Remark 5.5]{adler2009geometric}, a theory is rosy if and only if there is a \emph{strict independence relation}:  that is an $\mathrm{Aut}(\mathbb{M})$-invariant ternary relation on small subsets of $\mathbb{M}$ satisfying the properties listed in \cref{GS props}, plus symmetry, base monotonicity, and full existence.  Full existence is easily seen to be a consequence of extension and existence so follows from \cref{GS props} as well. Symmetry and base monotonicity follow from treelessness by \cref{sym+basemono}. 
\end{proof}

We end this section with the following general statement.
\begin{cor} \label{cor:alg base mon}
    For any theory $T$, $a \ind^{\GS}_A b$ if and only if $a \ind^{\GS}_{\acl(A)} b$.
\end{cor}
\begin{proof}
    First, assume $a \ind^{\mathrm{GS}}_{A} b$. By algebraicity and normality (see \cref{GS props}), $a \ind^{\GS}_A \acl(Ab)$. Now continue as in the proof of \cref{sym+basemono}, using \cref{prop:alg semidirect product} instead of \cref{useful prop} to get $a \ind^{\GS}_{\mathrm{acl}(A)} \acl(Ab)$. Finally, $a \ind^{\GS}_{\mathrm{acl}(A)} b$ follows from monotonicity.

    For the other direction, assume $a \ind^{\mathrm{GS}}_{\mathrm{acl}(A)} b$. Since $\acl(A) \ind^{\GS}_A b$ by algebraicity and left existence, we may apply left transitivity and monotonicity to get $a \ind_{A} b$. 
\end{proof}


\section{Stable theories are treeless} \label{Stable section}

In this section, we will prove that stable theories are treeless. This will involve an analysis of various indiscernible sequences living inside of treetop indiscernibles. We will work with treetop indiscernibles indexed by $\mathcal{T}$ and make use of the homogeneity of this structure. 

\begin{lem} \label{switcheroo1}
Assume $T$ is stable.  Suppose $(a_{\eta})_{\eta \in \omega^{\leq \omega}}$ is treetop indiscernible. Given any $\eta_{0} <_{lex} \ldots <_{lex} \eta_{n} \in \mathcal{T}^{+}$ for $n \geq 1$, there are $\nu_{0} <_{lex} \ldots <_{lex} \nu_{n-1}$ satisfying the following:
\begin{enumerate}
\item $(\nu_{0}, \ldots, \nu_{n-1}) \vDash \mathrm{qftp}_{L_{0,P}}(\eta_{0}, \ldots, \eta_{n-1}/\{\emptyset\})$.
\item $\eta_{n} \perp \left(\bigwedge_{j < n} \nu_{j} \right)$ and $\eta_{n} <_{lex} \left( \bigwedge_{ j < n} \nu_{j} \right)$.  
\item $(a_{\eta_{0}}, \ldots, a_{\eta_{n}}) \equiv (a_{\nu_{0}}, \ldots, a_{\nu_{n-1}}, a_{\eta_{n}})$.  
\end{enumerate}
\end{lem}

\begin{proof}
Suppose $\eta_{0} <_{lex} \ldots <_{lex} \eta_{n}$ is an arbitrary sequence from $\mathcal{T}^{+}$.  Let $\xi_{*} = \bigwedge_{j \leq n} \eta_{j}$ and let $\zeta_{*}$ be any element of $\mathcal{T}_{-}$ such that $\xi_{*} \vartriangleleft \zeta_{*} \vartriangleleft \eta_{n}$ and such that $\zeta_{*}$ is either strictly above or incomparable with each other element in the (finite) $L_{0,P}$-substructure of $\mathcal{T}$ generated by $\{\eta_{0}, \ldots , \eta_{n}\}$.  Choose some $\eta_{*} \in \mathcal{T}_{+}$ such that $\eta_{n} <_{lex} \eta_{*}$ and $\eta_{n} \wedge \eta_{*} = \zeta_{*}$.  

Now we choose a sequence of pairs of nodes $(\xi_{i},\zeta_{i})_{i \in \mathbb{Z}}$ satisfying the following:
\begin{enumerate}
\item For all $i < j$, 
$$
\xi_{i} \vartriangleleft \zeta_{i} \vartriangleleft \xi_{j} \vartriangleleft \zeta_{j} \vartriangleleft \eta_{*}.  
$$
\item For $i < 0$, $\zeta_{i} \vartriangleleft \zeta_{*}$ and, for $i \geq 0$, $\zeta_{*} \vartriangleleft \xi_{i}$.  
\end{enumerate}

For each $i \in \mathbb{Z}$, pick $(\eta_{i,0}, \ldots, \eta_{i,n-1})$ such that we have 
$$
(\eta_{i,0},\ldots, \eta_{i,n-1}, \xi_{i},\zeta_{i},\eta_{*}) \vDash \mathrm{qftp}_{L_{0,P}}(\eta_{0}, \ldots, \eta_{n-1},\xi_{*}, \zeta_{*}, \eta_{*}).
$$
Note that, by the choice of $\zeta_{i}$ and $\xi_{i}$, in fact, the sequence $(\overline{\eta}_{i})_{i \in \mathbb{Z}}$ is a quantifier-free indiscernible sequence, where $\overline{\eta}_{i} = (\eta_{i,0}, \ldots, \eta_{i,n-1})$.  Moreover, we have 
$$
(\overline{\eta}_{i},\eta_{n}) \vDash \mathrm{qftp}_{L_{0,P}}(\eta_{0}, \ldots, \eta_{n})
$$
for all $i < 0$, and 
$$
(\overline{\eta}_{i},\eta_{n}) \vDash \mathrm{qftp}_{L_{0,P}}(\overline{\eta}_{0},\eta_{n}),
$$  
for all $i \geq 0$.  

To conclude the proof, it suffices to show that $(a_{\eta_{0}}, \ldots, a_{\eta_{n}}) \equiv (a_{\eta_{0,0}}, \ldots, a_{\eta_{0,n-1}}, a_{\eta_{n}})$.  Suppose this is not true.  Then there is some formula $\varphi$ such that 
$$
\vDash \varphi(a_{\eta_{0}}, \ldots, a_{\eta_{n}}) \wedge \neg \varphi(a_{\eta_{0,0}}, \ldots, a_{\eta_{0,n-1}}, a_{\eta_{n}}).
$$
Then, by indiscernibility, we have 
$$
\{\varphi(a_{\eta_{i,0}}, \ldots, a_{\eta_{i,n-1}},x) : i < 0\} \cup \{ \neg \varphi(a_{\eta_{i,0}}, \ldots, a_{\eta_{i,n-1}},x) : i \geq 0\}
$$
is consistent, so $\varphi$ witnesses the order property in $T$, contradicting stability.  
\end{proof}

Recall that $\eta_{0}, \ldots, \eta_{n-1}$ are a \emph{fan} in a tree if there is a node $\nu$ such that $\eta_{i} \wedge \eta_{j} = \nu$ for all $i \neq j$. 

\begin{lem} \label{switcheroo2}
Assume $T$ is stable.  Suppose $(a_{\eta})_{\eta \in \mathcal{T}}$ is treetop indiscernible, $n \geq 1$, and $\eta_{0} <_{lex} \ldots <_{lex} \eta_{n}$ are from $\mathcal{T}^{+}$.  Then if $\eta_{1}, \ldots, \eta_{n}$ together form a fan with common meet $\zeta_{*}$ and $\eta_{0} \perp \zeta_{*}$, then there are $\nu_{1}, \ldots, \nu_{n}$ satisfying the following:
\begin{enumerate}
\item $(\nu_{1}, \ldots, \nu_{n}) \vDash \mathrm{qftp}_{L_{0,P}}(\eta_{1}, \ldots, \eta_{n})$. 
\item $\bigwedge_{1 \leq j \leq n} \nu_{j} <_{lex} \eta_{0}$.
\item $\eta_{0}, \nu_{1}, \ldots, \nu_{n}$ form a fan.
\item $(a_{\eta_{0}}, a_{\eta_{1}}, \ldots, a_{\eta_{n}}) \equiv (a_{\eta_{0}}, a_{\nu_{1}}, \ldots, a_{\nu_{n}})$.  
\end{enumerate}
\end{lem}

\begin{proof}
Let $\xi_{*} = \bigwedge_{j \leq n} \eta_{j}$ and let $\zeta_{*} = \bigwedge_{1 \leq j \leq n} \eta_{j}$.  By assumption, $\xi_{*} \vartriangleleft \zeta_{*}$.  Choose any $\eta_{*} \in \mathcal{T}_{+}$ with $\zeta_{*} \vartriangleleft \eta_{*}$ and $\eta_{n} <_{lex} \eta_{*}$.  Choose $(\xi_{i}, \zeta_{i})_{i \in \mathbb{Z}}$ such that
\begin{enumerate}
\item $\xi_{0} = \xi_{*}$ and $\zeta_{0} = \zeta_{*}$, and 
\item for all $i < j$, $\xi_{i} \vartriangleleft \zeta_{i} \vartriangleleft \xi_{j} \vartriangleleft \zeta_{j} \vartriangleleft \eta_{*}$. 
\end{enumerate}
From here, we follow the proof of \cref{switcheroo1}.  We pick, for each $i \in \mathbb{Z}$, some $(\eta_{i,1}, \ldots, \eta_{i,n})$ such that 
$$
(\eta_{i,1},\ldots, \eta_{i,n}, \xi_{i},\zeta_{i}, \eta_{*}) \vDash \mathrm{qftp}_{L_{0,P}}(\eta_{1}, \ldots, \eta_{n},\xi_{*}, \zeta_{*}, \eta_{*}).
$$
Then $(\overline{\eta}_{i})_{i \in \mathbb{Z}}$ is a quantifier-free indiscernible sequence, where $\overline{\eta}_{i} = (\eta_{i,1}, \ldots, \eta_{i,n})$, with
$$
(\eta_{0}, \overline{\eta}_{i}) \vDash \mathrm{qftp}_{L_{0,P}}(\eta_{0}, \ldots, \eta_{n})
$$
for all $i \leq 0$, and 
$$
(\eta_{0},\overline{\eta}_{i}) \vDash \mathrm{qftp}_{L_{0,P}}(\eta_{0}, \overline{\eta}_{0}),
$$  
for all $i \geq 0$.  

Then we define $\nu_{1}, \ldots, \nu_{n}$ by setting $\nu_{j} = \eta_{-1,j}$ for $1 \leq j \leq n$.  Condition (1) and (2) are clearly satisfied.  Note that $\nu_{1}, \ldots, \nu_{n}$ form a fan with common meet $\zeta_{-1}$.  Given any $j$ with $1 \leq j \leq n$, we also have $\nu_{j} \wedge \eta_{*} = \zeta_{-1}$ and since $\zeta_{-1} \vartriangleleft \xi_{0} \vartriangleleft \eta_{*}$, we have $\nu_{j} \wedge \xi_{0} = \zeta_{-1}$ and hence $\nu_{j} \wedge \eta_{0} = \zeta_{-1}$.  This shows that $\eta_{0}, \nu_{1}, \ldots, \nu_{n}$ form a fan, so condition (3) is satisfied as well.  

Finally, we check Condition (4).  Suppose this fails.  Then there is some formula $\varphi$ such that 
$$
\vDash \varphi(a_{\eta_{0}}, a_{\eta_{1}},\ldots, a_{\eta_{n}}) \wedge \neg \varphi(a_{\eta_{0}}, a_{\eta_{-1,1}} \ldots, a_{\eta_{-1,n}}).
$$
Then, by indiscernibility, we have 
$$
\{\varphi(x;a_{\eta_{i,1}}, \ldots, a_{\eta_{i,n}}) : i \leq 0\} \cup \{ \neg \varphi(x;a_{\eta_{i,1}}, \ldots, a_{\eta_{i,n}}) : i < 0\}
$$
is consistent, so $\varphi$ witnesses the order property in $T$, contradicting stability.  
\end{proof}

\begin{thm} \label{main stable thm}
Suppose $T$ is stable and $(a_{\eta})_{\eta \in \mathcal{T}}$ is a treetop indiscernible.  Then $(a_{\eta})_{\eta \in \mathcal{T}_{+}}$ is an indiscernible sequence (ordered by $<_{lex}$). 
\end{thm}

\begin{proof}
Note that if $\eta_{0} <_{lex} \ldots <_{lex} \eta_{n}$ and $\nu_{0} <_{lex} \ldots <_{lex} \nu_{n}$ are fans from $\omega^{\omega}$ then 
$$
\mathrm{qftp}_{L_{0,P}}(\overline{\eta}) = \mathrm{qftp}_{L_{0,P}}(\overline{\nu}).
$$
Therefore, in order to prove the theorem, it suffices to prove that if $\eta_{0} <_{lex} \ldots <_{lex} \eta_{n}$ is a sequence from $\omega^{\omega}$ and $\nu_{0} <_{lex} \ldots <_{lex} \nu_{n}$ is a fan in $\omega^{\omega}$, then
$$
(a_{\eta_{0}}, \ldots, a_{\eta_{n}}) \equiv (a_{\nu_{0}}, \ldots, a_{\nu_{n}}).  
$$
We will prove this by induction on $n$.  The $n = 0,1$ cases are trivial.  

Suppose now we are given $\eta_{0} <_{lex} \ldots <_{lex} \eta_{n+1}$ from $\mathcal{T}_{+}$.  By \cref{switcheroo1}, we find $\nu_{0}, \ldots, \nu_{n} \in \mathcal{T}_{+}$ such that 
\begin{enumerate}
\item $(\nu_{0}, \ldots, \nu_{n}) \vDash \mathrm{qftp}_{L_{0,P}}(\eta_{0}, \ldots, \eta_{n})$.
\item $\eta_{n+1} \perp \left(\bigwedge_{j \leq n} \nu_{j} \right)$ and $\eta_{n} <_{lex} \left( \bigwedge_{ j \leq n} \nu_{j} \right)$.  
\item $(a_{\eta_{0}}, \ldots, a_{\eta_{n}},a_{\eta_{n+1}}) \equiv (a_{\nu_{0}}, \ldots, a_{\nu_{n}}, a_{\eta_{n+1}})$.  
\end{enumerate}
Let $\eta_{*} = \eta_{n+1} \wedge \bigwedge_{j \leq n} \nu_{j}$ and set $\mathcal{T}' = \{\xi \in \mathcal{T} : \eta_{*} \vartriangleleft \xi, \eta_{n+1} <_{lex} \xi\}$.  Then $\nu_{0}, \ldots, \nu_{n} \in \mathcal{T}'$ and $(a_{\eta})_{\eta \in \mathcal{T}'}$ is a treetop indiscernible over $a_{\eta_{n+1}}$.  By induction, there is a fan $\eta'_{0} <_{lex} \ldots <_{lex} \eta'_{n}$ such that 
$$
(a_{\nu_{0}}, \ldots, a_{\nu_{n}}) \equiv_{a_{\eta_{n+1}}} (a_{\eta'_{0}}, \ldots, a_{\eta'_{n}}).  
$$
Then, by \cref{switcheroo2} applied to the tuple $(a_{\eta_{n+1}}, a_{\eta'_{0}}, \ldots, a_{\eta'_{n}})$, there are, in $\mathcal{T}$, $\eta''_{0} <_{lex} \ldots <_{lex} \eta''_{n} <_{lex} \eta_{n+1}$ such that $\eta''_{0}, \ldots, \eta''_{n}$, and $\eta_{n+1}$ form a fan and 
$$
(a_{\eta'_{0}}, \ldots, a_{\eta'_{n}}, a_{\eta_{n+1}}) \equiv (a_{\eta''_{0}}, \ldots, a_{\eta''_{n}}, a_{\eta_{n+1}}).  
$$
This yields 
$$
(a_{\eta_{0}}, \ldots, a_{\eta_{n}},a_{\eta_{n+1}}) \equiv (a_{\eta''_{0}}, \ldots, a_{\eta''_{n}}, a_{\eta_{n+1}}),
$$
as desired. \end{proof}

\begin{cor}
If $T$ is stable, then $T$ is treeless. 
\end{cor}

\begin{proof}
This follows from \cref{main stable thm} and \cref{no root}, since stable theories are NIP. 
\end{proof}

\section{From  \texorpdfstring{NSOP$_{1}$}{NSOP1} to simple} \label{simple section}

We will show in this section that treeless NSOP$_{1}$ theories are simple. We will show this, first, by analyzing $\ind^{\mathrm{GS}}$ in NSOP$_{1}$ theories, showing that it always agrees with Kim-independence in NSOP$_{1}$ theories with existence, and over models in all NSOP$_{1}$ theories.  We also give a rapid alternative proof in the special case of binary NSOP$_{1}$ theories, using the `lifting lemma' machinery from \cite[Section 6]{kaplan2019transitivity}.  

\subsection{Treeless NSOP$_{1}$ theories} \begin{defn}
\text{ }
\begin{enumerate}
\item We say $\varphi(x;y)$ has the \emph{tree property} if there is some $k < \omega$ and a collection of tuples $(a_{\eta})_{\eta \in \omega^{<\omega}}$ satisfying the following:
\begin{enumerate}
\item For all $\eta \in \omega^{\omega}$, $\{\varphi(x;a_{\eta | i}) : i < \omega\}$ is consistent.  
\item For all $\eta \in \omega^{<\omega}$, $\{\varphi(x;a_{\eta \frown \langle i \rangle}) : i < \omega\}$ is $k$-inconsistent.
\end{enumerate}
We say $T$ is \emph{simple} if no formula has the tree property modulo $T$.
\item We say $\varphi(x;y)$ has \emph{SOP}$_{1}$ if there is a collection of tuples $(a_{\eta})_{\eta \in 2^{<\omega}}$ satisfying the following:
\begin{enumerate}
\item For all $\eta \in 2^{\omega}$, $\{\varphi(x;a_{\eta | i}) : i < \omega\}$ is consistent. 
\item For all $\eta \perp \nu \in 2^{<\omega}$, if $\eta \unrhd (\eta \wedge \nu) \frown \langle 0 \rangle$ and $\nu = (\eta \wedge \nu) \frown \langle 1 \rangle$, then $\{\varphi(x;a_{\eta}), \varphi(x;a_{\nu})\}$ is inconsistent.  
\end{enumerate}
We say $T$ is NSOP$_{1}$ if no formula has SOP$_{1}$ modulo $T$.  
\end{enumerate}
\end{defn}

It is shown in \cite{kaplan2017kim} that, in any NSOP$_{1}$ theory, there is an independence relation $\ind^{K}$ called \emph{Kim-independence}, defined over models, that generalizes the familiar non-forking independence simple theories and has many nice properties.  Moreover, in simple theories, over models $\ind^{K}$ and $\ind^{f}$ agree \cite[Proposition 8.4]{kaplan2017kim}, where $\ind^{f}$ denotes non-forking independence. 

\begin{defn}
Suppose $M \vDash T$.  By an $\ind^{K}$\emph{-Morley sequence over} $M$, we mean an $M$-indiscernible sequence $I = \langle a_{i} : i < \omega \rangle$ such that $a_{i} \ind^{K}_{M} a_{<i}$ for all $i < \omega$. 
\end{defn}

\begin{fact} \label{Kim-independence facts}
Suppose $T$ is an NSOP$_{1}$ theory.  We have the following:
\begin{enumerate}
\item Symmetry:  if $M \vDash T$, $a \ind^{K}_{M} b$ if and only if $b \ind^{K}_{M} a$.  \cite[Theorem 5.16]{kaplan2017kim} 
\item $T$ is simple if and only if $\ind^{K}$ satisfies base monotonicity\textemdash that is, whenever $M \preceq N \vDash T$, if $a \ind^{K}_{M} Nb$, then $a \ind^{K}_{N} b$.   \cite[Proposition 8.8]{kaplan2017kim} 
\item Witnessing:  If $M \vDash T$ and $I = \langle a_{i} : i < \omega \rangle$ is an $\ind^{K}$-Morley sequence over $M$ with $a_{0} = a$, then $b \ind^{K}_{M} a$ if and only if there is $I' \equiv_{Ma} I$ such that $I'$ is $Mb$-indiscernible.  \cite[Theorem 5.1]{kaplan2019transitivity}
\item Lifting lemma:  If $M \preceq N \vDash T$ and $a \ind^{K}_{M} N$, then there is an $N$-indiscernible sequence $I = \langle a_{i} : i < \omega \rangle$ with $a_{0} = a$ that is both $\ind^{K}$-Morley over $M$ and $\ind^{K}$-Morley over $N$.  \cite[Proposition 3.3]{kaplan2019transitivity}
\end{enumerate}
\end{fact}

Note that the following proposition does not require treelessness:

\begin{prop} \label{treeless kimforking}
If $T$ is simple, then $\ind^{\mathrm{GS}}$ coincides with the usual non-forking independence. If $T$ is NSOP$_1$ and $M$ is a model, then $\ind^{\mathrm{GS}}_M$ coincides with Kim-independence over $M$. In particular, symmetry holds in these cases.
\end{prop}
\begin{proof}
We argue in the NSOP$_1$ case; the proof in the simple case is the same, except that we can drop the assumption that $M$ is a model (in that case Kim-independence is just the usual forking-independence). Let $M$ be a model and consider a tuple $b$. Let $p=\mathrm{tp}(b/M)$. Consider the partial type \[\pi(x) = p(x) \cup \{\neg \phi(x;c):c\in \mathbb{M}, \phi(x;c)\cup p(x)\text{ Kim-divides over }M\}.\] 
This partial type is $M$-invariant. To see that it is ind-definable over $M$, by \cref{lem_def}, we have to argue that the set 
$$
X = \{(a,\overline{b}) : \overline{b} \in \mathbb{M}^{\omega}, a \vDash \pi|_{M\overline{b}}\}
$$
is type-definable over $M$.  Fix $I = (a_{i})_{i < \omega}$, any coheir sequence over $M$ in $p$.  Let $q(x_{0},x_{1}, \ldots) = \text{tp}((a_{i})_{i < \omega}/M)$.  Notice that, if $(a,\overline{b}) \in X$, then, since $a \equiv_{M} a_{0}$, there is $I' = (a'_{i})_{i < \omega}$ with $I' \equiv_{M} I$ and $a'_{0} = a$.  By symmetry and the definition of $\pi$, $\overline{b} \ind^{K}_{M} a$ and hence there is $I'' \equiv_{Ma} I'$ which is $M\overline{b}$-indiscernible.  This shows that if $(a,\overline{b}) \in X$, then there is $I'' \vDash q$ which starts with $a$ and which is $M\overline{b}$ indiscernible.  On the other hand, if there is some $J \vDash q$ which starts with $a$ and which is $M\overline{b}$-indiscernible, then by symmetry and Kim's lemma, $a \ind^{K}_{M} \overline{b}$ so $(a,\overline{b}) \in X$.  This shows that $(a,\overline{b}) \in X$ if and only if 
$$
(\exists \overline{x}) \left[ \overline{x} \vDash q \wedge x_{0} = a \wedge \overline{x} \text{ is }M\overline{b}\text{-indiscernible}\right],
$$
which shows $X$ is type-definable over $M$.

Additionally, $\pi$ is generically stable:  if $\langle a_{i} : i < \omega \rangle$ is a sequence with $a_{i} \vDash \pi|_{M a_{<i}}$ for all $i < \omega$, then we have $a_{i} \ind^{K}_{M} a_{<i}$.  Suppose $\neg \varphi(x;b) \in \pi$, so $\varphi(x;b)$ Kim-divides over $M$ and we must show that $\vDash \neg \varphi(a_{i};b)$ for all but finitely many $i$.  If not, then, after throwing away a co-infinite set, we may assume $\vDash \varphi(a_{i};b)$ for all $i < \omega$. However, by symmetry, $\varphi(a_{0};y) \cup \mathrm{tp}_{y}(b/M)$ Kim-divides over $M$ and, thus, by \cite[Remark 5.3]{kaplan2019transitivity}, $\{\varphi(a_{i};y) : i < \omega\} \cup \text{tp}_{y}(b/M)$ is $k$-inconsistent for some $k$, a contradiction.  This establishes that $\pi$ is generically stable so $a \ind^{\mathrm{GS}}_{M} b$ implies $a \ind^{K}_{M} b$ (by symmetry and the definition of $\pi$). 

For the converse, assume that $a\ind^K_M b$ and for a contradiction that $a\nind^{\mathrm{GS}}_M b$. Set $p(x,y)=\mathrm{tp}(a,b/M)$. Let $\pi(y)$ be generically stable over $M$, consistent with $\mathrm{tp}(b/M)$, but not with $\mathrm{tp}(b/Ma)$. Take $(b_i:i<\omega)$ a $\ind^{\mathrm{GS}}$-Morley sequence in $\mathrm{tp}(b/M)$. Then by the first part of the proof, it is a $\ind^K$-Morley sequence. As $a\ind^K_M b$, $\bigwedge p(x,b_i)$ is consistent. This contradicts generic stability of $\pi$.
\end{proof}

%

\begin{thm} \label{nsop1 to simple}
Suppose $T$ is a treeless NSOP$_{1}$ theory.  Then $T$ is simple.  
\end{thm}

\begin{proof}
As $T$ is NSOP$_{1}$, we have, by \cref{treeless kimforking}, that $\ind^{K} = \ind^{\mathrm{GS}}$ over models.  By \cref{sym+basemono}, the treelessness of $T$ implies $\ind^{\mathrm{GS}}$ satisfies base monotonicity. 
\end{proof}

\begin{quest}
If $T$ is NSOP$_1$, is $\ind^{\mathrm{GS}}$ symmetric over an extension base, or even over an arbitrary base (see also \cref{que:symmetry})?
\end{quest}

\subsection{A quick alternate proof for binary NSOP$_{1}$ theories}

In this subsection, we give a short alternative proof that binary NSOP$_{1}$ theories are simple. This, of course, follows from \cref{nsop1 to simple} but admits a direct proof using established facts about Kim-independence.  The proof is short and different enough that we thought it worthwhile to include. 

\begin{lem} \label{sequence lemma}
Suppose $T$ is binary, $C$ is a set of parameters, and $I = \langle a_{i} : i < \omega \rangle$ and $J = \langle b_{i} : i < \omega \rangle$ are $C$-indiscernible sequences with $a_{0} = b_{0}$ and $I \equiv J$.  Then $I \equiv_{C} J$.   
\end{lem}

\begin{proof}
We may write $a_{i} = (a_{i,0},\ldots, a_{i,m-1})$ and likewise for $b_{i}$, for all $i < \omega$.  If $c \in C$, then, since $I$ and $J$ are $C$-indiscernible sequences starting with $a_{0} = b_{0}$, we have 
$$
a_{i,j}c \equiv a_{0,j}c \equiv b_{0,j}c \equiv b_{i,j}c,
$$
for all $i < \omega$ and $j < m$.  Since $I \equiv J$, it follows that any pair of elements selected from $IC$ will have the same type as the corresponding pair from $JC$, and thus, by binarity, $IC \equiv JC$.  This shows $I \equiv_{C} J$.  
\end{proof}

\begin{thm}
Suppose $T$ is binary and NSOP$_{1}$.  Then $T$ is simple.  
\end{thm}

\begin{proof}
By \cref{Kim-independence facts}(2), it suffices to show that $\ind^{K}$ satisfies base monotonicity.  So fix $M \prec N \vDash T$ and assume $a \ind^{K}_{M} Nb$.  We must show $a \ind^{K}_{N} b$.  By symmetry (\cref{Kim-independence facts}(1)), we have $bN \ind^{K}_{M} a$ and it suffices to establish $b \ind^{K}_{N} a$.  

As $a \ind^{K}_{M} N$, there is a sequence $I = \langle a_{i} : i < \omega \rangle$ with $a_{0} = a$ which is simultaneously $\ind^{K}$-Morley over $M$ and over $N$, by \cref{Kim-independence facts}(4).  Because $bN \ind^{K}_{M} a$, there is $J \equiv_{Ma} I$ such that $J$ is $Nb$-indiscernible.  By \cref{sequence lemma}, since $I \equiv_{M} J$ and $I$ and $J$ are both $N$-indiscernible, starting with $a$, we have $I \equiv_{N} J$, from which it follows that $J$ is $\ind^{K}$-Morley over $N$ as well.  By \cref{Kim-independence facts}(3), this shows $b \ind^{K}_{N} a$, completing the proof.  
\end{proof}

\section{From \texorpdfstring{NSOP$_{3}$}{NSOP3} to \texorpdfstring{NSOP$_{2}$}{NSOP2}} \label{nsop3 section}

In this section, we show that treeless NSOP$_{3}$ theories with trivial indiscernibility are NSOP$_{2}$.  Trivial indiscernibility is a weak form of binarity introduced in \cite{braunfeld2021characterizations}.  Because binary theories are always treeless, our results show, in particular, that binary NSOP$_{3}$ theories are necessarily NSOP$_{2}$.  

\begin{defn}
The properties SOP$_{2}$ and SOP$_{3}$ are defined as follows:
\begin{enumerate}
\item The theory $T$ has SOP$_{2}$ if there is a formula $\varphi(x;y)$ and a collection of tuples $(a_{\eta})_{\eta \in \omega^{<\omega}}$ satisfying the following:
\begin{itemize}
\item For all $\eta \perp \nu \in \omega^{<\omega}$, $\{\varphi(x;a_{\eta}), \varphi(x;a_{\nu})\}$ is inconsistent.
\item For all $\eta \in \omega^{\omega}$, $\{\varphi(x;a_{\eta | i }): i < \omega\}$ is consistent.
\end{itemize}
\item The theory $T$ has SOP$_{3}$ if there are formulas $\psi_{0}(x;y)$ and $\psi_{1}(x;y)$ satisfying the following:
\begin{itemize}
\item For all $k < \omega$, 
$$
\{\psi_{0}(x;a_{i}) : i \leq k\} \cup \{\psi_{1}(x;a_{j}) : j > k\}
$$
is consistent. 
\item For all $i < j$, $\{\psi_{1}(x;a_{i}), \psi_{0}(x;a_{j})\}$ is inconsistent. 
\end{itemize}
\end{enumerate}
\end{defn}

\begin{rem}
We have defined SOP$_{2}$ and SOP$_{3}$ in the form most convenient for us to use them, though the equivalence of SOP$_{3}$ defined here with its usual definition can be found in \cite[Claim 2.19]{shelah1995toward}. 
\end{rem}

In the end, we did not use the following description of SOP$_{3}$, but we found the reformulation of SOP$_{3}$ in terms of detecting disjointness of intervals to be useful at the level of intuition and so decided to include it.  

\begin{lem}
Let $\mathcal{I}$ denote the set of all non-empty closed intervals in $[0,1]$.  The following are equivalent:
\begin{enumerate}
\item $T$ has SOP$_{3}$.
\item There is a formula $\varphi(x;y)$ and a collection of tuples $(b_{I})_{I \in \mathcal{I}}$ such that, for any family $\mathcal{J} \subseteq \mathcal{I}$ consisting of intervals with all endpoints distinct, 
$$
\{\varphi(x;b_{I}) : I \in \mathcal{J}\} \text{ is consistent }\iff \bigcap \mathcal{J} \neq \emptyset.  
$$
\end{enumerate}
\end{lem}

\begin{proof}
(1)$\implies$(2).  By compactness, there are formulas $\varphi(x;y)$ and $\psi(x;y)$ and an indiscernible sequence $(a_{i})_{i \in [0,1]}$ such that:
\begin{itemize}
\item For all $k \in [0,1]$, 
$$
\{\varphi(x;a_{i}) : i \in [0,k] \} \cup \{\psi(x;a_{i}) : i \in (k,1]\}
$$
is consistent.
\item For all $i < j \in [0,1]$, 
$$
\{\psi(x;a_{i}), \varphi(x;a_{j})\}
$$
is inconsistent.  
\end{itemize}
Define a formula $\chi(x;y,z) = \varphi(x;y) \wedge \psi(x;z)$.  For each $I = [i,j] \in \mathcal{I}$, let $b_{I} = (a_{i},a_{j})$.  Suppose $\mathcal{J} = \{I_{\alpha} = [i_{\alpha},j_{\alpha}] : \alpha < \beta \} \subseteq \mathcal{I}$ is a family of intervals with all endpoints distinct. If $\bigcap \mathcal{J} \neq \emptyset$, then, for all $\alpha, \alpha' < \beta$, we have $i_{\alpha} < j_{\alpha'}$, and hence 
$$
\{\varphi(x;a_{i_{\alpha}}) : \alpha < \beta\} \cup \{\psi(x;a_{j_{\alpha}}) : \alpha < \beta\}
$$
is consistent.  It follows that $\{\chi(x;b_{I}) : I \in \mathcal{J}\}$ is consistent.  

Conversely, if $\bigcap \mathcal{J}  = \emptyset$, then there are disjoint closed intervals $I,I' \in \mathcal{J}$.  Without loss of generality, $I = [i,j]$, $I' = [i',j']$ and $j < i'$.  Then $\{\psi(x;a_{j}),\varphi(x;a_{i'})\}$ is inconsistent, from which it follows that $\{\chi(x;b_{I}), \chi(x;b_{I'})\}$ is inconsistent.  

(2)$\implies$(1).  For each $k \in (0,\frac{1}{3})$, let $c_{k} = (b_{[\frac{1}{3}+k,\frac{2}{3}+k]},b_{[k,\frac{1}{3}+k]})$ and define $\psi_{0}(x;y,z) = \varphi(x;y)$ and $\psi_{1}(x;y,z) = \varphi(x;z)$.  Notice, then, that if $k \in (0,\frac{1}{3})$, 
$$
\bigcap_{i \in (0,k]}\left[\frac{1}{3}+i,\frac{2}{3}+i\right] \cap \bigcap_{i \in (k,\frac{1}{3})} \left[i,\frac{1}{3}+i\right] \neq \emptyset,
$$
so $\{\psi_{0}(x;c_{i}) : i \in (0, k]\} \cup \{\psi_{1}(x;c_{i}) : i \in (k,\frac{1}{3})\}$ is consistent.  On the other hand, if $i < j$ are numbers in $(0,\frac{1}{3})$, then $[i,\frac{1}{3}+i] \cap [\frac{1}{3}+j,\frac{2}{3}+j] = \emptyset$ so $\{\psi_{0}(x;c_{j}), \psi_{1}(x;c_{i})\}$ is inconsistent.  Thus, we obtain SOP$_{3}$.  
\end{proof}

\begin{lem} \label{indiscernibility}
Suppose $T$ is treeless.  Suppose $(a_{\eta})_{\eta \in \mathcal{T}}$ is treetop indiscernible and $\nu \in \mathcal{T}_{-}$.  Let 
\begin{eqnarray*}
J &=& \{\eta \in \mathcal{T}_{+} : \eta \wedge \nu \vartriangleleft \nu, \eta <_{lex} \nu\} \\
J' &=& \{\eta \in \mathcal{T}_{+} : \eta \wedge \nu \vartriangleleft \nu, \nu <_{lex} \eta\}
\end{eqnarray*}  
Then $(a_{\eta})_{\eta \in J}$ and $(a_{\eta})_{\eta \in J'}$ are $a_{\nu}$-indiscernible sequences (with $J$ and $J'$ linearly ordered by $<_{lex}$).  
\end{lem}

\begin{proof}
By symmetry, it suffices to prove that $(a_{\eta})_{\eta \in J}$ is $a_{\nu}$-indiscernible. 
As $\mathcal{T}_{-}$ is the Fra\"iss\'e limit of finite meet trees (as $L_{0}$-structures) it is $\aleph_{0}$-saturated, and thus there is an $L_{0}$-embedding $f: \mathcal{T} \to \mathcal{T}_{-}$ sending some element from $\mathcal{T}_+$ above $\nu$ in the tree partial order to $\nu$. 
Let $\mathcal{S}$ be the image of $f$, and set $\mathcal{S}_{+} = f(\mathcal{T}_{+})$. Interpret $P^{\mathcal{S}} = \mathcal{S}_{+}$, so that $\mathcal{S} \cong \mathcal{T}$  as $L_{0,P}$-structures. Note that $\mathcal{S}$ is an $L_{0}$-substructure of $\mathcal{T}$, but it is \emph{not} an $L_{0,P}$-substructure of $\mathcal{T}$. 


For each element $\eta \in \mathcal{S}_{+}$, pick some $\zeta(\eta) \in \mathcal{T}_{+}$ with $\eta \unlhd \zeta(\eta)$.  Define $(b_{\eta})_{\eta \in \mathcal{S}}$ as follows:  for $\eta \in \mathcal{S}_{-}$, we set $b_{\eta} = a_{\eta}$ and, for $\eta \in \mathcal{S}_{+}$, we set $b_{\eta} = (a_{\eta},a_{\zeta(\eta)})$.  

\begin{claim} \label{cla:new tree is treetop ind}
    $(b_{\eta})_{\eta \in \mathcal{S}}$ is treetop indiscernible.
\end{claim}

\begin{proof}[Proof of Claim]
    Suppose $\overline{\eta}$ and $\overline{\xi}$ are finite tuples from $\mathcal{S}$ with $\mathrm{qftp}_{L_{0,P}}(\overline{\eta}) = \mathrm{qftp}_{L_{0,P}}(\overline{\xi})$.  We may assume that $\overline{\eta} = (\overline{\eta}_{0}, \overline{\eta}_{1})$ and $\overline{\xi} = (\overline{\xi}_{0}, \overline{\xi}_{1})$ with $\overline{\eta}_{0},\overline{\xi}_{0} \in \mathcal{S}_{-}$ and $\overline{\eta}_{1}, \overline{\xi}_{1} \in \mathcal{S}_{+}$.  Let $\overline{\eta}' = (\overline{\eta}_{0}, \overline{\eta}_{1}, \zeta(\overline{\eta}_{1}))$ and, likewise, $\overline{\xi}' = (\overline{\xi}_{0}, \overline{\xi}_{1}, \zeta(\overline{\xi}_{1}))$, finite tuples from $\mathcal{T}$.  Our assumption that $\mathrm{qftp}_{L_{0,P}}(\overline{\eta}) = \mathrm{qftp}_{L_{0,P}}(\overline{\xi})$ in $\mathcal{S}$ entails that $\mathrm{qftp}_{L_{0,P}}(\overline{\eta}') = \mathrm{qftp}_{L_{0,P}}(\overline{\xi}')$ in $\mathcal{T}$, and therefore that $a_{\overline{\eta}'} \equiv a_{\overline{\xi}'}$.  By the definition of $(b_{\eta})_{\eta \in \mathcal{S}}$, it follows that $b_{\overline{\eta}} \equiv b_{\overline{\xi}}$, which proves the claim.
\end{proof}

Next, we establish the following:

\begin{claim}\label{cla:witnesses in new tree}
    For all finite $\overline{\eta}$ from $J$, there is some $\overline{\xi}$ from $\mathcal{S}_{+}$ such that
    $$
    \mathrm{qftp}_{L_{0,P}}(\overline{\eta},\nu) = \mathrm{qftp}_{L_{0,P}}(\zeta(\overline{\xi}),\nu)
    $$
    holds in $\mathcal{T}$.
\end{claim}

\begin{proof}[Proof of Claim]
The proof uses the following easy observation.

\noindent\hypertarget{guys on top}{$(\dagger)$} In $\mathcal{T}$, if $a_0<_{lex} \ldots <_{lex} a_{k-1}$, $a_i, a_j$ are not comparable in the tree partial order for distinct $i,j$ and and $a_i \trianglelefteq b_i$ for all $i<k$, then $\mathrm{qftp}_{L_{0}}(\overline{a}) = \mathrm{qftp}_{L_{0}}(\overline{b})$. (This is true since $b_i \wedge b_j = a_i \wedge a_j$ for any $i,j<k$.)

Let $\eta' \in \mathcal{T}_+$ be such that $f(\eta') = \nu$. Recall that $f$ was chosen so that $\eta' \trianglerighteq \nu$. By choice of $f$, $\mathrm{qftp}_{L_{0}}(\overline{\eta},\eta') = \mathrm{qftp}_{L_{0}}(f(\overline{\eta}),\nu)$. By \hyperlink{guys on top}{$(\dagger)$}, this type is equal to  $\mathrm{qftp}_{L_{0}}(\overline{\eta},\nu)$ on the one hand (since $\nu \trianglelefteq \eta'$) and to  $\mathrm{qftp}_{L_{0}}(\zeta(f(\overline{\eta})),\nu)$ on the other hand (since in general, $\eta \trianglelefteq \zeta(\eta)$). This gives the desired equality of types without $P$, but on both generated structures, the only elements from $P$ are $\overline{\eta}$ and $\zeta(\overline{\eta})$. Together, we are done.
\end{proof}
By treelessness and \cref{cla:new tree is treetop ind}, it follows that $(b_{\eta})_{\eta \in \mathcal{S}_{+}}$ is an indiscernible sequence.  By definition, this entails that $(a_{\zeta(\eta)})_{\eta \in \mathcal{S}_{+}, \eta <_{lex} \nu}$ is an $a_{\nu}$-indiscernible sequence.  Then, by treetop indiscernibility and \cref{cla:witnesses in new tree}, it follows that $(a_{\eta})_{\eta \in J}$ is $a_{\nu}$-indiscernible.  
\end{proof}

We will mostly make use of a certain corollary of \cref{indiscernibility}, but, in order to state it, we will need the following definition from \cite{braunfeld2021characterizations}: 

\begin{defn}
Say $T$ has \emph{indiscernible triviality} if, whenever $I = \langle a_{i} : i < \omega \rangle$ is simultaneously $a$-indiscernible and $b$-indiscernible, then $I$ is $ab$-indiscernible.  
\end{defn}

We note that binary theories clearly have indiscernible triviality, though there are nonbinary examples. 

\begin{cor} \label{indiscernibility cor}
Suppose $T$ is treeless and has indiscernible triviality.  Suppose $(a_{\eta})_{\eta \in \mathcal{T}}$ is treetop indiscernible and $\nu \in \mathcal{T}_{-}$.  Let 
\begin{eqnarray*}
J &=& \{\eta \in \mathcal{T}_{+} : \eta \wedge \nu \vartriangleleft \nu, \eta <_{lex} \nu\} \\
J' &=& \{\eta \in \mathcal{T}_{+} : \eta \wedge \nu \vartriangleleft \nu, \nu <_{lex} \eta\}
\end{eqnarray*}  
Then $(a_{\eta})_{\eta \in J}$ and $(a_{\eta})_{\eta \in J'}$ are $a_{\unrhd \nu}$-indiscernible sequences (with $J$ and $J'$ linearly ordered by $<_{lex}$).  
\end{cor}

\begin{proof}
Fix $\nu$ as in the statement.  For each $\xi \in \mathcal{T}_{-}$, define $J_{\xi}$ so that 
$$
J_{\xi} = \{\eta \in \mathcal{T}_{+} : \eta \wedge \xi \vartriangleleft \nu, \eta <_{lex} \xi\}.
$$
As in the proof of \cref{indiscernibility}, it is enough to prove that $(a_{\eta})_{\eta \in J_{\nu}}$ is $a_{\unrhd \nu}$-indiscernible, by symmetry.  Note that if $\nu' \unrhd \nu$, then $J_{\nu'} \supseteq J_{\nu}$ and $(a_{\eta})_{\eta \in J_{\nu'}}$ is $a_{\nu'}$-indiscernible by \cref{indiscernibility} and thus, \emph{a fortiori}, $(a_{\eta})_{\eta \in J_{\nu}}$ is $a_{\nu'}$-indiscernible.  Moreover, by treelessness, $(a_{\eta})_{\eta \in \mathcal{T}_{+}}$ is an indiscernible sequence so $(a_{\eta})_{\eta \in J_{\nu}}$ is indiscernible over $(a_{\eta})_{\eta \in C(\nu)}$.  It follows by indiscernible triviality that $(a_{\eta})_{\eta \in J_{\nu}}$ is $a_{\unlhd \nu}$-indiscernible. 
\end{proof}

\begin{lem} \label{witnessing lemma}
Assume $T$ has SOP$_{2}$ witnessed by the formula $\varphi(x;y)$.  Then there is a treetop indiscernible $(a_{\eta})_{\eta \in \omega^{\leq \omega}}$ satisfying the following:
\begin{itemize}
\item For all $\eta \perp \nu$ from $\omega^{<\omega}$, $\{\varphi(x;a_{\eta}), \varphi(x;a_{\nu})\}$ are inconsistent.
\item For all $\eta \in \omega^{\omega}$, $a_{\eta} \vDash \{\varphi(x;a_{\eta | i }) : i < \omega\}$.
\end{itemize}
\end{lem}

\begin{proof}
Let $(b_{\eta})_{\eta \in \omega^{<\omega}}$ be a tree of tuples witnessing that $\varphi$ has SOP$_{2}$, i.e. 
\begin{enumerate}
\item For all $\eta \perp \nu$, $\{\varphi(x;b_{\eta}), \varphi(x;b_{\nu})\}$ is inconsistent.
\item For all $\eta \in \omega^{\omega}$, $\{\varphi(x;b_{\eta | i}) : i < \omega\}$ is consistent.
\end{enumerate}
Choose, for each $\eta \in \omega^{\omega}$, some $b_{\eta} \vDash \{\varphi(x;b_{\eta | i }) : i< \omega\}$.  Let $(a_{\eta})_{\eta \in \omega^{\leq \omega}}$ be any treetop indiscernible locally based on $(b_{\eta})_{\eta \in \omega^{\leq \omega}}$.  It is easy to check that this satisfies the desired properties. 
\end{proof}

\begin{thm} \label{nsop3 to nsop2}
Suppose $T$ is a treeless theory with indiscernible triviality.  Then if $T$ has SOP$_{2}$, then $T$ has SOP$_{3}$.  
\end{thm}

\begin{proof}
Assume $T$ has SOP$_{2}$, witnessed by the formula $\varphi(x;y)$.  Then, by \cref{witnessing lemma} and compactness, we can find a treetop indiscernible $(a_{\eta})_{\eta \in \mathcal{T}}$ satisfying the following:
\begin{enumerate}
\item If $\eta \perp \nu$ are from $\mathcal{T}_{-}$, then $\{\varphi(x;a_{\eta}), \varphi(x;a_{\nu})\}$ is inconsistent. 
\item If $\eta^{*} \in \mathcal{T}_{+}$, then $a_{\eta^{*}} \vDash \{\varphi(x;a_{\eta}) : \eta \vartriangleleft \eta^{*}\}$.  
\end{enumerate}  
Let $\psi_{0}(x;y,z) = \varphi(x;y)$ and $\psi_{1}(x;y,z) = \varphi(x;z)$.  We will show that $\psi_{0}$ and $\psi_{1}$ witness SOP$_{3}$.  

By compactness, it suffices to show, for each $n$, that there is a sequence $(d_{i})_{i < n}$ such that 
\begin{enumerate}
\item $\{\psi_{0}(x;d_{i}) : i \leq j \} \cup \{\psi_{1}(x;d_{i}) : j < i < n\}$ is consistent for all $j \leq n$.  
\item If $i < j < n$, then $\{\psi_{1}(x;d_{i}), \psi_{0}(x;d_{j})\}$ is inconsistent. 
\end{enumerate}  
So fix an arbitrary $n \geq 1$.  Choose arbitrary $\eta \perp \nu$ in $\mathcal{T}_{-}$ with $\eta <_{lex} \nu$.  We choose $\eta^{*}_{l,0} <_{lex} \eta^{*}_{r,0}$ in $\mathcal{T}_{+}$ with $\eta^{*}_{l,0} \wedge \eta^{*}_{r,0} = \eta$ and, likewise, $\nu^{*}_{l,0} <_{lex} \nu^{*}_{r,0}$ in $\mathcal{T}_{+}$ with $\nu^{*}_{l,0} \wedge \nu^{*}_{r,0} = \nu$.  

Now we choose $\nu^{*}_{l,1}, \nu^{*}_{r,1}, \ldots , \nu^{*}_{l,n-1}, \nu^{*}_{r,n-1} \in \mathcal{T}_{+}$ such that 
$$
\nu^{*}_{l,0} <_{lex} \nu^{*}_{l,1} \ldots <_{lex} \nu^{*}_{l,n-1} <_{lex} \nu^{*}_{r,0} <_{lex} \nu^{*}_{r,1} <_{lex}  \ldots <_{lex} \nu^{*}_{r,n-1}. 
$$
We define some intervals in $\mathcal{T}_{+}$ as follows:
$$
I_{0} = \{\xi \in \mathcal{T}_{+} : \eta^{*}_{l,0} <_{lex} \xi <_{lex} \eta^{*}_{r,0}\},
$$
and, for all $i < n$, 
$$
J_{i} = \{\xi \in \mathcal{T}_{+} : \nu^{*}_{l,i} <_{lex} \xi <_{lex} \nu^{*}_{r,i}\}.
$$
Then, finally, we define 
$$
K = \{\xi \in \mathcal{T}_{+} : \xi \wedge \nu \vartriangleleft \nu, \xi <_{lex} \nu\}.  
$$
\begin{claim}
    There are $\sigma_{0}, \ldots, \sigma_{n-1} \in \mathrm{Aut}(\mathcal{T}_{+}, <_{lex})$ (where $(\mathcal{T}_{+}, <_{lex})$ is regarded as a dense linear order with no additional structure) satisfying the following:
    \begin{enumerate}
        \item $\sigma_{i}(J_{0}) = J_{i}$ for all $i < n$.  
        \item $\nu^{*}_{l,i} \in \sigma_{i+1}(K)$ for all $i < n-1$.
        \item The map $(a_{\xi})_{\xi \in \mathcal{T}_{+}} \mapsto (a_{\sigma_{i}(\xi)})_{\xi \in \mathcal{T}_{+}}$ is partial elementary over $a_{\eta}$ for all $i < n$.  
    \end{enumerate}
\end{claim}

\begin{proof}[Proof of Claim]
    To begin, we define $\sigma_{0}$ to be the identity map.  Assume $\sigma_{0}, \ldots, \sigma_{i}$ have been defined.  Write $\mathcal{T}_{+}$ as the disjoint union $L_{0} \cup L_{1}$ where 
    $$
    L_{0} = \{\xi \in \mathcal{T}_{+} : (\exists \eta')[\eta \unlhd \eta' \wedge  \xi <_{lex} \eta']\}.
    $$
    Then it is easy to see that $L_{0}$ and $L_{1}$ are both countable dense linear orders without endpoints and $L_{0} \subsetneq K$.  Pick some $\zeta \in K \setminus L_{0}$ and some $\zeta' \in \{\xi \in \mathcal{T}_{+} : \nu^{*}_{l,i} <_{lex} \xi <_{lex} \nu^{*}_{l,i+1}\}$.  Define $\tau_{0} \in \mathrm{Aut}((L_{0},<_{lex}))$ to be the identity and $\tau_{1} \in \mathrm{Aut}((L_{1},<_{lex}))$ to be an automorphism mapping $(\zeta,\nu^{*}_{l,0},\nu^{*}_{r,0}) \mapsto (\zeta',\nu^{*}_{l,i+1},\nu^{*}_{r,i+1})$.  Then define $\sigma_{i+1} = \tau_{0} \cup \tau_{1}$, which is an automorphism of $(\mathcal{T}_{+}, <_{lex})$ with $\sigma_{i+1}(J_{0}) = J_{i+1}$.  Moreover, since $K$ is an initial segment of $\mathcal{T}_{+}$, it follows $\sigma_{i+1}(K)$ is an initial segment of $\mathcal{T}_{+}$.  Since $\sigma_{i+1}(K)$ also must contain $\zeta'$ and $\zeta' >_{lex} \nu^{*}_{l,i}$, we must also have $\nu^{*}_{l,i} \in \sigma_{i+1}(K)$.  Finally, we know by treelessness that $(a_{\xi})_{\xi \in L_{1}}$ is $(a_{\xi})_{\xi \in L_{0}}$-indiscernible and also $a_{\unrhd \eta}$-indiscernible, by \cref{indiscernibility cor}, as we have 
    $$
    L_{1} = \{\xi \in \mathcal{T}_{+} : \xi \wedge \eta \vartriangleleft \eta, \eta <_{lex} \xi\}.  
    $$
    It follows by indiscernible triviality that $(a_{\xi})_{\xi \in L_{1}}$ is $(a_{\xi})_{\xi \in L_{0}}a_{\unrhd \eta}$-indiscernible.  Thus the mapping $(a_{\xi})_{\xi \in \mathcal{T}_{+}} \mapsto (a_{\sigma_{i+1}(\xi)})_{\xi \in \mathcal{T}_{+}}$ is partial elementary over $a_{\eta}$.
\end{proof}

Now we pick $\eta^{*}_{l,i}, \eta^{*}_{r,i} \in \sigma_{i}(K)$ such that the $(\eta^{*}_{l,i})_{i < n}$ and $(\eta^{*}_{r,i})_{i < n}$ are increasing and, moreover, $\eta^{*}_{l,i} <_{lex} \eta^{*}_{r,0} <_{lex} \nu^{*}_{l,i-1} 
<_{lex} \eta^{*}_{r,i}$ for each $1 \leq i < n$. 
Note that, since $\sigma_{i}(J_{0}) = J_{i}$, we also have $\eta^{*}_{r,i} <_{lex} \nu^{*}_{l,i}$.  Define 
$$
I_{i} = \{\xi \in \mathcal{T}_{+} : \eta^{*}_{l,i} <_{lex} \xi <_{lex} \eta^{*}_{r,i}\}.
$$
Note that we have 
$$
\bigcap_{k \leq i} J_{i} \cap \bigcap_{i < k < n} I_{i} \supseteq (\nu^*_{l,i},\eta^*_{r,i+1}) \neq \emptyset
$$
for all $i < n$, where $\eta^*_{r,n} = \nu^*_{r,0}$ and 
$$
I_{i} \cap J_{j} = \emptyset
$$
for all $i \leq j < n$.  

For each $i < n$, let $\tilde{\sigma_{i}} \in \mathrm{Aut}(\mathbb{M}/a_{\eta})$ extend the mapping $(a_{\xi})_{\xi \in \mathcal{T}_{+}} \mapsto (a_{\sigma_{i}(\xi)})_{\xi \in \mathcal{T}_{+}}$ with $\tilde{\sigma_{0}}$ defined to be the identity.  Define $a_{0} = a_{\nu}$ and $a_{i} = \tilde{\sigma_{i}}(a_{\nu})$ for $1 \leq i < n$.  By \cref{indiscernibility}, we have 
$$
\{a_{\xi} : \xi \in K\}
$$
is $a_{0}$-indiscernible and contains $(a_{\xi})_{\xi \in I_{0}}$.  It follows, then, that for each $i < n$, 
$$
\tilde{\sigma_{i}}\left( \{a_{\xi} : \xi \in K\} \right) = \{a_{\sigma_{i}(\xi)} : \xi \in K\} = \{a_{\xi} : \xi \in \sigma_{i}(K)\}
$$
is $a_{i}$-indiscernible and contains $(a_{\xi})_{\xi \in I_{0} \cup \ldots \cup I_{i}}$.  Since we have $K = \sigma_{0}(K) \subseteq \sigma_{1}(K) \subseteq \ldots \subseteq \sigma_{n-1}(K)$, we have, by indiscernible triviality, that $(a_{\xi})_{\xi \in \sigma_{i}(K)}$ is $a_{i}\ldots a_{n-1}$-indiscernible.  Given $1 \leq i < n$, we can find some $\tau_{i} \in \mathrm{Aut}((\mathcal{T}_{+},<_{lex}))$ which restricts to an automorphism of $\sigma_{i}(K)$ taking $I_{0}$ to $I_{i}$ and which is the identity on $\mathcal{T}_{+} \setminus \sigma_{i}(K)$.  Then the mapping $(a_{\xi})_{\xi \in \mathcal{T}_{+}} \mapsto (a_{\tau_{i}(\xi)})_{\xi \in \mathcal{T}_{+}}$ is partial elementary over $a_{i}\ldots a_{n-1}$ so we can find some extension $\tilde{\tau_{i}} \in \mathrm{Aut}(\mathbb{M}/a_{i}\ldots a_{n-1})$ mapping $(a_{\xi})_{\xi \in I_{0}} \mapsto (a_{\xi})_{\xi \in I_{i}}$.  We define $b_{0} = a_{\eta}$ and $b_{i} = \tilde{\tau_{i}}(b_{0})$ for each $1 \leq i < n$.  

This completes the construction, so now we check that it works. Note that, by construction, if $i \leq j$, then 
$$
b_{i}a_{j} \equiv a_{\eta}a_{j} \equiv a_{\eta}a_{\nu}
$$ 
and hence $\{\varphi(x;b_{i}), \varphi(x;a_{j})\}$ is inconsistent, by the definition of SOP$_{2}$.  On the other hand, we know that 
$$
\bigcap_{k \leq i} J_{i} \cap \bigcap_{i < k < n} I_{i} \neq \emptyset
$$
so we can fix some $\xi^{*}$ in this intersection.  Then for each $i < k < n$, we know that $\tau^{-1}_{k}(\xi^{*}) \in I_{0}$ and hence $\vDash \varphi(a_{\tau^{-1}_{k}(\xi^{*})},a_{\eta})$, which implies $\vDash \varphi(a_{\xi^{*}},b_{k})$.  Additionally, for each $k \leq i$, we know $\sigma^{-1}_{k}(\xi^{*}) \in J_{0}$ and hence $\vDash \varphi(a_{\sigma^{-1}_{k}(\xi^{*})},a_{\nu})$, which entails $\vDash \varphi(a_{\xi^{*}},a_{k})$.  This shows 
$$
a_{\xi^{*}} \vDash \{\varphi(x;a_{k}) : k \leq i\} \cup \{\varphi(x;b_{k}) : i < k < n\}.
$$
Therefore, defining $d_{i} = (a_{i},b_{i})$ for all $i < n$, we have proved that $\psi_{0}$ and $\psi_{1}$ have SOP$_{3}$.  
%
\end{proof}

\begin{cor}
If $T$ is a treeless NSOP$_{3}$ theory with indiscernible triviality, then $T$ is simple.  In particular, a binary NSOP$_{3}$ theory is simple. 
\end{cor}

\begin{proof}
By \cref{nsop3 to nsop2}, such $T$ is NSOP$_{2}$ and, by \cite{mutchnik2022nsop}, this entails that $T$ is NSOP$_{1}$ which, in turn, entails that $T$ is simple by \cref{nsop1 to simple}. The `in particular' clause follows because binarity implies treelessness and indiscernible triviality. 
\end{proof}

\subsection*{Acknowledgements}

We would like to thank Alex Kruckman for catching an error with a preliminary version of our characterization of SOP$_{3}$ in terms of interval intersections.

\bibliographystyle{alpha}
\bibliography{ms.bib}{}

\end{document}